\documentclass[11pt]{article}
\usepackage{amsmath,mathrsfs,amsfonts,amsthm,amscd,amssymb,graphicx,color,cite}
\numberwithin{equation}{section}
\usepackage{txfonts} 
\usepackage[titletoc,title]{appendix}
\usepackage{enumitem}

\newcommand{\dist}{\text{dist}} 

\newcommand{\diam}{\text{diam}}
\renewcommand{\div}{\text{div}\,}

\newcommand{\osc}{\text{osc}}
\newcommand{\R}{{\mathbb R}} \newcommand{\K}{{\mathbb K}}

\newcommand{\N}{{\mathbb N}}

\newcommand{\e}{\varepsilon} 

\newcommand{\F}{\mathbf{F}}

\newcommand{\A}{\mathbf{A}}

\newcommand{\ba}{\mathbf{a}}

\newcommand{\M}{{\mathrm M}}
\newcommand{\calM}{{\mathcal M}}

\def\longequals{\mathbin{=\kern-2pt=}}
\def\eqdef{\mathbin{\buildrel \rm def \over \longequals}}

\newtheorem{theorem}{Theorem}[section]
\newtheorem{corollary}[theorem]{Corollary}
\newtheorem{definition}[theorem]{Definition}
\newtheorem{remark}[theorem]{Remark}
\newtheorem{lemma}[theorem]{Lemma}
\newtheorem{proposition}[theorem]{Proposition}

\numberwithin{equation}{section}

\newcommand{\beq}{\begin{equation}}
\newcommand{\eeq}{\end{equation}}

\definecolor{darkred}{rgb}{.70,.12,.20}

\definecolor{darkgreen}{rgb}{.20,.52,.14}

\title{\vspace{-1in}  Boundary regularity for  quasilinear elliptic equations with general Dirichlet boundary data
}
\author{ Truyen  Nguyen\footnotemark[1]}

\date{}
\providecommand{\subjclass}[1]{{\textit{2010 Mathematics Subject Classification.}} #1}
\providecommand{\keywords}[1]{{\textit{Key words and phrases.}} #1}

\begin{document}
\maketitle
\begin{abstract}
We study global regularity  for solutions of  quasilinear  elliptic equations of the form $\div \A(x,u,\nabla u) = \div \F $ in 
rough domains $\Omega$ in $\R^n$ with nonhomogeneous Dirichlet boundary condition. The vector field $\A$ is assumed to be continuous in $u$, and its growth in $\nabla u$ is like that of 
the $p$-Laplace operator. We establish  global gradient estimates in weighted Morrey spaces  for  weak solutions $u$ to the 
equation under the Reifenberg flat condition for $\Omega$, 
a small BMO condition in $x$ for $\A$, and  an optimal condition for the Dirichlet boundary data.  
\end{abstract}
\\

\renewcommand{\thefootnote}{\fnsymbol{footnote}}

\footnotetext[1]{Department of Mathematics,
The University of Akron, 302 Buchtel Common, Akron, OH 44325, USA. Email: tnguyen@uakron.edu.
The research of the  author is supported in part
by a grant  from the Simons Foundation (\# 318995) and by the UA Faculty Research grant FRG
(\# 207359)} 






\setcounter{equation}{0}

\section{Introduction}\label{sec:Intro}
We investigate global gradient estimates  for weak solutions to  the  Dirichlet problem
\begin{equation}\label{ME}
 \left\{
 \begin{alignedat}{2}
  &\div  \A(x,  u,  \nabla u)=  \div \F \quad \mbox{ in}\quad \Omega,\\\
&u=\psi \qquad\qquad\qquad \qquad \mbox{on}\quad \partial\Omega, 
 \end{alignedat} 
  \right.
\end{equation}
when $\Omega$ is a bounded domain in $\R^n$ ($n\geq 2$) and the vector field $\A$ is only continuous in the $u$ variable and 
possibly discontinuous in the $x$ variable. Let  $\K\subset \R$ be an open interval and consider the general vector field
\[
\A = \A(x,z,\xi) : \Omega\times \overline\K\times \R^n \longrightarrow \R^n
\]
which is  a Carath\'eodory map, that is, $\A(x,z,\xi)$ is measurable in $x$ for every $(z,\xi)\in \overline \K \times \R^n$ and 
continuous in $(z,\xi)$ for a.e. $x$. We assume that $\xi \mapsto\A(x,z,\xi)$ is differentiable on 
$ \mathbb{R}^n\setminus \{0\}$ for a.e. $x$ and all $z\in \overline \K$. Also,  there exist constants  $\Lambda>0$,  $1< p<\infty$, 
and a nondecreasing and right continuous function $\omega: [0,\infty)\to [0,\infty)$ with $\omega(0)=0$
such that  the following  conditions are satisfied  for a.e. $x\in \Omega$ and all $z\in \overline\K$:
\begin{align}
&\langle \partial_\xi \A(x,z,\xi) \eta, \eta\rangle \geq \Lambda^{-1} |\xi|^{p-2} |\eta|^2\,  \qquad\qquad \qquad \qquad\forall \xi\in \R^n\setminus \{0\} \mbox{ and }\forall \eta\in\R^n,\label{structural-reference-1}\\ 
& |\A(x,z,\xi)|  + |\xi|\,  | \partial_\xi \A(x,z,\xi) | \leq \Lambda |\xi|^{p-1}\qquad\qquad\qquad \forall\xi\in  \R^n,\label{structural-reference-2}\\
& |\A(x,z_1,\xi)-\A(x,z_2,\xi)|  \leq \Lambda |\xi|^{p-1} \omega(|z_1 - z_2|) \qquad\quad \,\forall z_1, z_2\in \overline\K\mbox{ and }
\forall \xi\in \R^n. \label{structural-reference-3}
\end{align}

Interior $C^{1,\alpha}$ theory for the homogeneous equation associated to \eqref{ME} 
was established by DiBenedetto \cite{D} and Tolksdorf \cite{T}  extending the celebrated  H\"older gradient estimates by   
Ural\'tceva \cite{Ur} and  Uhlenbeck \cite{Uh}
for the homogeneous $p$-Laplace equation. Furthermore,  Lieberman \cite{Li} derived the   global $C^{1,\alpha}$ estimates for bounded weak
solutions of the corresponding  Dirichlet boundary problem when both the domain and the boundary data are of class $C^{1,\alpha}$.
On the other hand, interior $W^{1,q}$ estimates  for  the nonhomogeneous  quasilinear equation
\eqref{ME}  was investigated in \cite{NP} using the perturbation method by 
Caffarelli-Peral \cite{CP} together with the two-parameter scaling technique  introduced  in  \cite{HNP1} to deal with a specific 
parabolic equation.
 When $\A$  has sufficiently small BMO oscillation in $x$ and  is Lipschitz continuous in the $z$ variable, 
it was established in  \cite{NP} that: $|\F|^{\frac{1}{p-1}}\in L^q_{loc}  \Longrightarrow \nabla u\in L^q_{loc}$ for any $q>p$.  This result 
was extended in \cite{BPS,N,Ph1,Ph2} to cover more general situations. In particular, 
the authors of \cite{BPS}  derived the corresponding global estimates for Reifenberg flat domains and for zero Dirichlet boundary data.  Moreover,  they  were able to weaken the condition on $\A$ by 
allowing only H\"older continuity in the  $z$ variable. In a recent paper \cite{DiN} with Di Fazio,  we obtained interior gradient estimates in generalized  weighted Morrey spaces for solutions of \eqref{ME} 
when $\A$ is merely  continuous   in $z$. This, in particular, extends
the gradient estimates in the classical  Morrey spaces  obtained in  \cite{Ra,MP2} for the case  $\A(x,z,\xi)=\A(x,\xi)$.

Our purpose of the current work is two folds. On one hand, we 
 extend the mentioned  result in  \cite{BPS} to general and optimal  Dirichlet boundary condition. On the  other hand, we develop the boundary counterpart of the interior estimates in  \cite{DiN} by 
 deriving global  gradient  estimates in generalized weighted Morrey spaces for  bounded  solutions of  \eqref{ME} in Reifenberg flat domains.  These two goals are treated  in  a unified manner  and our achieved results give a comprehensive  picture of  gradient estimates for equation \eqref{ME}.
   In what follows we consider a bounded domain with its  boundary
 being   flat in the following sense of Reifenberg \cite{Re}.
\begin{definition}\label{def:Re} Let $\delta,\, R>0$. 
A bounded domain $\Omega\subset \R^n$  ($n\geq 2$)   is said to be $(\delta,R)$-Reifenberg flat if for every  $\bar x\in \partial\Omega$ and 
every $\rho\in (0,R)$ there is a local
coordinate system $\{x_1,\dots, x_n\}$ with origin at  the point $\bar x$ and such that 
\begin{align*}
 B_\rho(\bar x) \cap \{x:\, x_n > \rho \delta \} \subset  B_\rho(\bar x)\cap \Omega \subset B_\rho(\bar x) \cap \{x:\, x_n > -\rho \delta \}.
\end{align*}
\end{definition}
 The Reifenberg flatness means that the boundary $\partial\Omega$  is well approximated by hyperplanes at every point and at every scale.
We note that  a domain $\Omega$ is  Reifenberg flat if its boundary is $C^1$ smooth or, more generally, 
its boundary  is locally given as the graph of a Lipschitz continuous function with small Lipschitz constant.  However, the
class of Reifenberg flat domains is much larger  and contains domains with rough fractal
boundaries (see \cite{To}). 
In order to state our main results, let us recall the so-called Muckenhoupt class of $A_s$ weights. By definition,  a weight is  a nonnegative  locally  integrable function on $\R^n$ that is positive almost everywhere.
A weight $w$ belongs to the  class $A_s$, $1<s< \infty$, if 
\[
 [w]_{A_s} :=  \sup{ \Big(\fint_{B} w(x)\, dx\Big) \Big(\fint_{B} w(x)^{\frac{-1}{s-1}}\, dx\Big)^{s-1}  } <\infty,
\]
where the supremum is  taken over all balls $B$ in $\R^n$.  
We also say that  $w$ belongs to the class $A_\infty$  if 
\[
 [w]_{A_\infty}  :=  \sup{ \Big(\fint_{B} w(x)\, dx\Big) \exp{\Big(\fint_{B} \log{ w(x)^{-1}}\, dx\Big)}  } <\infty.
\]
Now let $U\subset \R^n$ be a bounded open set, $w$ be a weight, $1\leq q <\infty$, and $\varphi$ be a positive function on the set of nonempty
open balls in $\R^n$. 
A function $g: U\to \R$ is said to 
belong to the weighted space $L^q_w(U)$ if
\[
 \|g\|_{L^q_w(U)} := \Big(\int_{U} |g(x)|^q w(x)\, dx\Big)^{\frac1q}<\infty.
\]
We define the weighted Morrey 
space $\calM^{q,\varphi}_w(U)$  to be the set of all functions $g\in L^q_w(U)$ satisfying
\begin{equation}\label{weighted-morrey}
\|g\|_{\calM^{q,\varphi}_w(U)} := \sup_{ \bar x\in U,\, 0<r\leq \diam(U)} \Bigg(\frac{\varphi(B_r(\bar x))}{ w(B_r(\bar x))} \int_{B_r(\bar x)\cap U}
|g(x)|^q w \, dx \Bigg)^{\frac1q}<\infty.
\end{equation}
Notice that $\calM^{q,\varphi}_w(U)=L^q_w(U)$ if  $\varphi(B)=  w(B)$, and  we obtain the classical 
Morrey space $\calM^{q,\lambda}(U)$ ($0\leq \lambda\leq n$) by taking $w=1$ and $\varphi(B)= |B|^{\frac{\lambda}{n}}$. For brevity,
the Morrey space $\calM^{q,\varphi}_w(U)$ with $w=1$ will be  denoted by 
$\calM^{q,\varphi}(U)$.  Let us next recall the class $\mathcal B_\alpha$ for $\varphi$ that was introduced in \cite{DiN}.
\begin{definition}\label{class-B} Let $\varphi$ be a positive function on the set of nonempty
open balls in $\R^n$. We say that $\varphi$ belongs to the class $\mathcal B_\alpha$ with $\alpha\geq 0$ if there exists  $C>0$ such that
\[
  \frac{\varphi(B_r(x))}{\varphi(B_s(x))}\leq C \Big(\frac{r}{s}\Big)^\alpha
\]
for every $x\in\R^n$ and every $0<r\leq s$. We define $\mathcal B_+ :=\cup_{\alpha>0} \mathcal B_\alpha \subset \mathcal B_0$.
 \end{definition}
 
Throughout the paper the conjugate exponent of a number $l\in (1,\infty)$ is denoted by 
$l'$.  We will also adopt the following notation:
$\Omega_r(y):=\Omega\cap B_r(y)$, $\Omega_r :=\Omega_r(0)$, and 
 $\langle\A\rangle_{\Omega_r(y)}(z,\xi) := \frac{1}{|B_r(y)|}\int_{\Omega_r(y)} \A(x,z,\xi)\, dx$.
 On the other hand, $\bar g_E :=\frac{1}{|E|}\int_{E} g(x)\, dx$ whenever $g\in L^1(E)$ with 
$E\subset \R^n$ being  a measurable and bounded  set.
For a fixed number $M>0$, we  define 
\begin{equation}\label{def:Theta}
\Theta_{\Omega_r(y)}(\A):= \sup_{z\in \overline\K \cap [-M, M]} 
\frac{1}{|B_r(y)|}\int_{\Omega_r(y)}\Big[\sup_{ \xi\neq 0}\frac{|\A(x,z,\xi) -
\langle\A\rangle_{\Omega_r(y)}(z,\xi)|}{|\xi|^{p-1}}\Big]\, dx.
\end{equation}
Our first  main result is:
\begin{theorem}\label{thm:weighted-morrey} 
Let $\A$ satisfy 
\eqref{structural-reference-1}--\eqref{structural-reference-3} with $p>1$,  and let $w$ be an $A_s$ weight for some $1<s<\infty$.
Then for any $q\geq p$,  $M>0$, and $\varphi\in \mathcal B_+$ with $\sup_{x\in \Omega} \varphi( B_{\diam(\Omega)}(x))<\infty$, 
there exists a constant   $\delta=\delta( p, q,  n,\omega,  \Lambda, M,s, [w]_{A_s})>0$  such that: 
if $\Omega$ is  $(\delta,R)$-Reifenberg flat,
\begin{equation}\label{smallness-1}
\sup_{0<r< R}\sup_{y\in\overline{\Omega}}  \Theta_{\Omega_r(y)}(\A)\leq \delta,
\end{equation}
and $u$  is a weak solution of 
\eqref{ME} satisfying $\|u\|_{L^\infty(\Omega)}+ \|\psi\|_{L^\infty(\Omega)}\leq M$, we have
\begin{equation}\label{weighted-morrey-est}
\| \nabla u \|_{\calM^{q,\varphi}_w(\Omega)} \leq    C\Bigg(
\|\nabla u\|_{L^p(\Omega)} +  \| \M_{\Omega}(|\psi|^{p} + |\F|^{p'})\|_{\calM^{1, \varphi^{\frac{p}{q}}}(\Omega)}^{\frac{1}{p}} 
+ \|\M_{\Omega} (|\psi|^{p} +|\F|^{p'})\|_{\calM^{\frac{q}{p},\varphi}_w(\Omega)}^{\frac{1}{p}}\Bigg).
\end{equation}
Here $\M_{\Omega}$ denotes  the centered Hardy--Littlewood   maximal operator (see Definition~3.1 in \cite{DiN}), and  $C>0$ 
is a constant depending only on  $q$, $p$, $n$,  $\omega$,  $\Lambda$, $M$, $\varphi$, $s$, $R$, $\diam(\Omega)$, and $[w]_{A_s}$.
\end{theorem}
The above theorem holds true for any weight $w$ in the class $A_\infty$. When $q>p$ and certain  additional information about the weights and $ \varphi, \, \phi$ is 
given, we can further estimate the two quantities in \eqref{weighted-morrey-est} involving the maximal function
of $|\F|^{p'}$  to obtain:
\begin{theorem}[global weighted Morrey space estimate]\label{thm:global-weighted-morrey} Let   $\A$ satisfy 
\eqref{structural-reference-1}--\eqref{structural-reference-3} with $p>1$.
Let  $q>p$, $ w\in A_\infty$, $v\in A_{\frac{q}{p}}$, $\varphi\in \mathcal B_+$ with $\sup_{x\in \Omega} \varphi( B_{\diam(\Omega)}(x))<\infty$,
and $\phi\in \mathcal B_0$  satisfy
\begin{align}
 & [w, v^{1-(\frac{q}{p})'}]_{A_{\frac{q}{p}}}:= \sup_{B}{ \Big(\fint_{B} w\, dx\Big) \Big(\fint_{B} v^{1-(\frac{q}{p})'}\, dx\Big)^{\frac{q}{p}-1}  } 
  <\infty,\label{cond:wv}\\
 & \frac{v(2 B)}{w(2 B)} \frac{1}{\phi(2 B)} \leq C_* \, \frac{1}{\varphi(B)} 
\quad \mbox{for all balls  $B\subset \R^n$}.\label{cond:vwvarphi} 
 \end{align}
Then for any $M>0$, there exists a small constant 
$\delta=\delta( p, q,  n, \omega, \Lambda, M, [w]_{A_\infty})>0$  such that: if $\Omega$ is  $(\delta,R)$-Reifenberg flat,
\eqref{smallness-1} holds,
and $u$  is a weak solution of 
\eqref{ME} satisfying $\|u\|_{L^\infty(\Omega)}+ \|\psi\|_{L^\infty(\Omega)}\leq M$, we have
\begin{equation}\label{eq:neat-global-est}
\| \nabla u \|_{\calM^{q,\varphi}_w(\Omega)}  \leq    C\Big\|  |\nabla\psi| + |\F|^{\frac{1}{p-1}}\Big\|_{\calM^{q,\phi}_v(\Omega)}
\end{equation}
with  $C$  depending only on  $q$, $p$, $n$, $\omega$,  $\Lambda$, $M$,   $\varphi$, $\phi$, $C_*$, $R$, $\diam(\Omega)$,  $[w]_{A_\infty}$,  $[v]_{A_{\frac{q}{p}}}$, and 
 $[w, v^{1-(\frac{q}{p})'}]_{A_{\frac{q}{p}}}$.
\end{theorem}
This result complements the global $C^{1,\alpha}$ regularity  developed  by Lieberman \cite{Li} for the corresponding homogeneous
equation when $(x, z)\mapsto \A(x,z,\cdot)$ is $C^\alpha$  and both $\partial \Omega$ and  $\psi$ belong to the class $C^{1,\alpha}$.
The boundedness assumption for $u$ in Theorems~\ref{thm:weighted-morrey} and \ref{thm:global-weighted-morrey}
is used to handle the dependence of the principal part on the 
solution itself. A simple inspection of our proofs reveals that this condition is not needed in the special case  $\A(x,z,\xi)=\A(x,\xi)$.
As far as we know all available global $W^{1,q}$ estimates for nonlinear equation of $p$-Laplacian type are only  established for identically zero Dirichlet boundary data. Our  boundary data  is general and the imposed condition on $\psi$ is optimal in view of the  linear case. We expect that our method of dealing with  nonhomogeneous boundary data can be useful for   other nonlinear elliptic  equations.

The central point in  proving the above main results is to be able to 
show that gradients of weak  solutions $u$ to equation \eqref{ME}  can be approximated
in an invariant way by  bounded gradients in $L^p$ norm (see
Proposition~\ref{lm:localized-compare-gradient}). In order to achieve this we encounter 
four  main difficulties: the discontinuity of $\A(x,z,\xi)$ in the $x$ variable and the dependence of $\A$ on the $z$ variable, the 
nonhomogeneous Dirichlet boundary data,
the roughness of the domain,  and   
the fact that  equations of the  form  \eqref{ME} are not  invariant with respect to dilations and rescaling of 
domains.
Using the method  in  \cite{CC,CP} and an argument in \cite{AM,BPS,NP}, we deal with the first two issues 
in Lemmas~\ref{lm:step1} and \ref{lm:step2} by 
freezing the $x$ and $z$ variables and 
comparing  solution $u$ of \eqref{ME}  to that of the corresponding frozen equation with zero Dirichlet boundary data.  
The third issue is handled  by a compactness argument allowing us to extract information from the limiting equations whose  domains 
have  flat boundaries, 
see the proofs of Lemma~\ref{lm:step3} and Lemma~\ref{lm:compactness}.
To overcome the last 
issue about the lack of the invariant structure, we use the key idea introduced in \cite{HNP1,NP} by  enlarging  the class of equations under 
consideration in a suitable way. Precisely, 
we  consider the associated quasilinear elliptic  equations with two parameters, i.e.  equation \eqref{GE} below.
The chief advantage  of working with \eqref{GE} is that 
equations of this form  are invariant with respect to dilations and rescaling 
of domains. However, there arise new difficulties in dealing with 
the parameters. It is essential for the success of  our analysis   that  
all the  obtained approximation estimates and involving  constants  must be independent of the two parameters. The large part of this paper is devoted to achieving this.

\bigskip

The organization of the paper is as follows. We recall some preliminary results
in Subsections~\ref{sub:weight}--\ref{sub:classical}  and state two key regularity results (Theorems~\ref{thm:main} and \ref{thm:localized-main})
in Subsection~\ref{sub:key-reg}
about gradient estimates in weighted $L^q$ spaces.  Section~\ref{approximation-gradient} is devoted to proving
Proposition~\ref{lm:localized-compare-gradient}
which shows  that gradients of weak solutions to  two-parameter equation  \eqref{GE} can be approximated by bounded gradients in a small neighborhood
of any point in the domain. Using this crucial result,
we establish  in Subsection~\ref{sub:density-est} some density estimates 
for gradients and then prove Theorems~\ref{thm:main} and \ref{thm:localized-main} in Subsection~\ref{sub:Lebesgue-Spaces}. 
Finally, the main results stated in Theorems~\ref{thm:weighted-morrey}
 and \ref{thm:global-weighted-morrey} are derived in Subsection~\ref{sec:Morrey-Spaces} as consequences of Theorem~\ref{thm:main}.
 
\section{Preliminaries and  key regularity results}\label{sec:preliminary}
\subsection{Basic properties of $A_s$ weights and  estimates for the maximal function}\label{sub:weight}
 Given a weight $w$ and a measurable set $E\subset \R^n$, we use the notation $dw(x)= w(x)\, dx$ and  $w(E)=\int_E w(x)\, dx$.

\begin{lemma}\label{weight:basic-pro} 
Let $w\in A_s$ for some $1< s <\infty$. Then
there exist $0<\beta \leq 1$ and $K>0$ depending only on $n$ and $[w]_{A_s}$ such that 
\begin{equation}\label{strong-doubling}
[w]_{A_s}^{-1}\, \Big(\frac{|E|}{|B|}\Big)^s \leq  \frac{w(E)}{w(B)}\leq K\,  \Big(\frac{|E|}{|B|}\Big)^\beta
\end{equation}
  for all balls $B$ and all measurable sets $E\subset  B$. In particular, $w$ is doubling with
   $w(2B)\leq 2^{ns} [w]_{A_s}  w(B)$.
 \end{lemma}
  
 \begin{lemma}[Characterizations of $A_\infty$ weights] Suppose that $w$ is a weight. Then $w$ is in $A_\infty$ 
 if and only if there exist $0<A, \, \nu<\infty$ such that  for all balls $B$ and all measurable sets $E\subset  B$  we have
\begin{equation}\label{charac-A-infty}
\frac{w(E)}{w(B)}\leq A \Big(\frac{|E|}{|B|}\Big)^{\nu}.
\end{equation}
When $w\in   A_\infty$,  the above constants $A$ and $\nu$ depend only on $n$ and  $[w]_{A_\infty}$. Conversely, given constants $A$ and $\nu$  satisfying  \eqref{charac-A-infty},  we have 
$[w]_{A_\infty}\leq C(n,A,\nu)$.
\end{lemma}

Let $\tilde \M$ denotes  the uncentered Hardy--Littlewood   maximal operator (see Definition~3.1 in \cite{DiN}). 
For two weights $w_1$ and $w_2$, let
\begin{align*}
  [w_1, w_2]_{A_q}
  &:= \sup_{B}{ \Big(\fint_{B} w(x)\, dx\Big) \Big(\fint_{B} w_2\, dx\Big)^{q-1}  },\\
   [w_1, w_2]_{S_q}   &:= \sup_{B}\Bigg(\frac{1}{w_2(B)}\int_{B}[ \tilde\M(w_2 \chi_B)]^q w_1\, dx \Bigg)^{\frac1q}. 
 \end{align*}
Then  $[w_1, w_2]_{A_q}\leq [w, w_2]_{S_q}^q$.
Moreover, we have the following estimates from \cite{DiN} for the maximal function in weighted Morrey spaces. 
 \begin{lemma}[Corollary~3.6 in \cite{DiN}] \label{lm:two-weight-case} Let $w,\, v$ be two weights and $\varphi,\, \phi$ be two positive functions on the set of nonempty open balls in  $\R^n$.
Let $1<q<\infty$ and $U\subset \R^n$ be a bounded open set. Assume that  $w$ is doubling and there exists a constant $C_*>0$ such that
\begin{equation}\label{wmu1-vmu2-sufficient}
\sup_{2r \leq s\leq 2\diam(U)}  \frac{v(B_s(y))}{w(B_s(y))} \frac{1}{\phi(B_s(y))} \leq C_* \, \frac{1}{\varphi(B_r(y))} \quad \mbox{for all $y\in U$
and $0<r\leq \diam(U)$}. 
 \end{equation}
Assume in addition that one of the following two conditions is satisfied:
\begin{enumerate}
 \item[({\bf A})] There exists $r>1$ such that
 $\displaystyle\quad 
   \sup_{B}{ \Big(\fint_{B} w\, dx\Big) \Big(\fint_{B} v^{r(1-q')}\, dx\Big)^{\frac{q-1}{r}}  }<\infty$.
  \item[({\bf B})] $[w, v^{1-q'}]_{A_q}<\infty$
and $v^{1-q'}\in A_\infty$. 
\end{enumerate}
Then $[w, v^{1-q'}]_{S_q}<\infty$ and there exists a constant $C>0$ depending only on $n$, $q$, $C_*$, the doubling constant for $w$, and 
$[w, v^{1-q'}]_{S_q}$  such that
\begin{equation*}
 \|\tilde\M_U(f)\|_{\calM^{q,\varphi}_w(U)}\leq C \|f\|_{\calM^{q,\phi}_v(U)}
 \qquad \forall f\in L^1(U).
 \end{equation*}
\end{lemma}

\begin{lemma}[Corollary~3.7 in \cite{DiN}] \label{lm:one-weight-case}  Let $1<q<\infty$, $w\in A_q$, and $U\subset \R^n$ be a bounded open set.
Assume that $\varphi$ and $\phi$ are two positive functions on the set of  open balls in  $\R^n$
such that there exists  $C_*>0$ satisfying
\begin{equation*}\label{mu1-mu2-sufficient}
\sup_{2r \leq s\leq 2\diam(U)}  
\phi(B_s(y))^{-1} \leq C_* \, \varphi(B_r(y))^{-1} \quad \mbox{for all $y\in U$
and $0<r\leq  \diam(U)$}. 
 \end{equation*}
Then there exists  a constant $C>0$ depending only on $n$, $q$, $C_*$, and $[w]_{A_q}$ such that
\begin{equation*}
 \|\tilde\M_U(f)\|_{\calM^{q,\varphi}_w(U)}\leq C \|f\|_{\calM^{q,\phi}_w(U)}\quad\mbox{for any } f\in L^1(U).
 \end{equation*}
\end{lemma}

\subsection{An energy inequality and some classical regularity estimates}\label{sub:classical} 
Let us consider the following equation
\begin{equation}\label{GE}
 \left\{
 \begin{alignedat}{2}
  &\div  \Big[\frac{ \A(x, \lambda\theta  u, \lambda \nabla u)}{\lambda^{p-1}}\Big]=  \div \F \quad \mbox{in}\quad \Omega,\\\
&u=\psi \qquad\qquad\qquad \qquad \qquad\mbox{ on}\quad \partial\Omega,
 \end{alignedat} 
  \right.
\end{equation}
where $\lambda,\, \theta>0$ are two parameters.
We will use the fact (see  the proof of  \cite[Lemma~1]{T}) that condition \eqref{structural-reference-1}
  implies that  
  \begin{align}\label{structural-consequence}
&\big\langle  \A(x,z,\xi) -\A(x,z,\eta), \xi-\eta\big\rangle \geq 
\left \{
\begin{array}{lcll}
  4^{1-p}\Lambda^{-1} |\xi-\eta|^p &\text{if}\quad p\geq 2,\\
4^{-1}\Lambda^{-1} \big(|\xi| +|\eta|)^{p-2} |\xi-\eta|^2 &\quad\,\,\text{ if}\quad 1 <p<2
\end{array}\right. 
\end{align}
for a.e. $x\in \Omega$, all $z\in \overline\K$, and all $\xi,\, \eta\in \R^n$.

\begin{proposition}[energy estimate]\label{prop:energy}
Let $\psi \in W^{1,p}(\Omega)$ and $u$ be a weak solution of \eqref{GE}. Then we have 
\begin{equation}\label{eq:energy}
\int_{\Omega} |\nabla u|^p\, dx \leq C(p,n,\Lambda)   \int_{\Omega} \big(|\nabla \psi|^p+ |\F|^{p'}\big)\, dx.
\end{equation}
\end{proposition}
\begin{proof} 
Let $\tilde \A(x, z,\xi) := \frac{\A(x,\lambda \theta z, \lambda \xi)}{\lambda^{p-1}}$. Then 
by using $u-\psi$ as a  test function in equation \eqref{GE} we get
\[
\int_{\Omega} \langle \tilde\A(x,   u, \nabla u), \nabla u  \rangle \, dx
=\int_{\Omega} \langle \tilde\A(x,   u, \nabla u),  \nabla \psi \rangle \, dx + \int_{\Omega} \langle\F, \nabla u -\nabla \psi \rangle \, dx. 
\]
But it follows from \eqref{structural-consequence} for $\eta=0$ and the fact $\A(x,z,0)=0$ that $\langle \A(x,z,\xi), \xi\rangle \geq C(p,\Lambda) |\xi|^p$. Therefore, we obtain 
\begin{align*}
\int_{\Omega} | \nabla u|^p \, dx
\leq C \Bigg[\int_{\Omega} |\nabla u|^{p-1} |\nabla \psi|  \, dx + \int_{\Omega} |\F| |\nabla u |  \, dx
+\int_{\Omega} |\F|  |\nabla \psi|  \, dx \Bigg].
\end{align*}
We deduce from  this and Young's inequality  that estimate \eqref{eq:energy} holds true.
\end{proof}

Let $\ba = \ba(x,\xi) : \Omega\times  \R^n \longrightarrow \R^n$ 
be  measurable in $x$ for every $\xi\in  \R^n$ and 
continuous in $\xi$ for a.e. $x\in \Omega$. In addition, we assume that there exist constants  $\Lambda>0$ and  $1< p<\infty$
such that the following  structural conditions 
are satisfied  for a.e. $x\in \Omega$ and all $\xi\in\R^n$:
\begin{equation}\label{simple-structural}
\langle  \ba(x,\xi), \xi\rangle \geq \Lambda^{-1} |\xi|^{p}\quad \mbox{ and }\quad 
 |\ba(x,\xi)|  \leq \Lambda |\xi|^{p-1}.
\end{equation}
The following interior H\"older estimate is classical, see for instance \cite{Gi,MZ}.
  \begin{theorem}[interior H\"older estimate]\label{classical-Holder} 
Let  $\Omega\subset \R^n$ be  a bounded domain  and  $\ba(x,\xi)$ satisfy 
\eqref{simple-structural}. 
Suppose $w\in W^{1,p}_{loc}(\Omega)$ is a weak solution
of $\div \ba(x,\nabla w)=0$ in $\Omega$.
Then $u$ is continuous in $\Omega$ and has the following bound on its modulus of continuity:
if $B_R(\bar x)\subset \Omega$ and $r\in (0, R)$, then we have 
\begin{align*}
\displaystyle
 \underset{B_r(\bar x)}{\osc} w 
\leq C  \big(\frac{r}{R}\Big)^\beta \underset{B_R(\bar x)}{\osc} w,
\end{align*} 
where $C>0$ and $\beta\in (0,1)$ depend only on $p$, $n$, 
 and $\Lambda$.
\end{theorem}
\begin{proof}
This  is a special case of Theorem~4.11 in \cite{MZ}. Notice that due to a much more general equation in
\cite{MZ} their corresponding constants $C$ and $\beta$ depend also on $\|w\|_{L^\infty(B_R(\bar x))}$. However, an 
inspection of their proof in page 196 reveals that with structural condition \eqref{simple-structural}
these constants can be chosen to depend only on $p$, $n$, 
 and $\Lambda$ as stated.
\end{proof}
The next regularity result is well known  and  is a special case  of  \cite[Theorem~1.1]{KK} and 
 \cite[Theorem~4.19 and Corollary~4.20]{MZ}  (see also \cite{Tr} and \cite[Theorem~6.8 and estimate~(7.54)]{Gi}).
\begin{theorem}[global higher integrable and H\"older estimates]\label{thm:inter-regularity} 
Let  $\Omega\subset \R^n$ ($n\geq 2$) be  a bounded domain,   $\ba(x,\xi)$ satisfy 
\eqref{simple-structural},  $x_0\in \partial\Omega$,  and $r_0>0$.
Assume that there exist positive constants $ c_*$ and $\rho_*$ satisfying
\begin{equation}\label{uniform-p-thick}\big| B_\rho(z) \setminus 
\Omega \big| \geq  c_* \, |B_\rho(z)|
\end{equation}
for all $z\in B_{r_0}(x_0)\cap \partial
\Omega$ and all $\rho\in (0, \rho_*)$.  Suppose that   $w\in W^{1,p}(\Omega_{r_0}(x_0))$ is  a weak solution
of 
\begin{equation*} 
 \left\{
 \begin{alignedat}{2}
  &\div \,\ba(x, \nabla w)=0\quad \mbox{in}\quad \Omega_{r_0}(x_0),\\\
&w=0 \qquad\qquad\quad \mbox{ on}\quad B_{r_0}(x_0) \cap\partial\Omega.
 \end{alignedat} 
  \right.
\end{equation*}
Then
\begin{enumerate}
 \item[(i)]   There exist constants $p_0\in (p,\infty)$ and $C>0$ depending only on $p$, $n$, $c_*$, 
 and $\Lambda$ such that: if $0<r< s\leq r_0$ and 
  $B_{s}(y)\subset \Omega_{r_0}(x_0)$,   we have
\begin{align*}
&\Big(\frac{1}{|B_r(y)|}\int_{\Omega_r(y)} |\nabla w|^{p_0} 
dx\Big)^{\frac1p_0} \leq  C 
\, \Big(\frac{1}{|B_s(y)|} \int_{\Omega_s(y)} |\nabla w|^{p} dx\Big)^{\frac1p}.
\end{align*} 
\item[(ii)] For any $z\in B_{r_0} (x_0)\cap \partial\Omega$ and any $r\in (0,r_0]$, we have 
\begin{align*}
\displaystyle
\underset{\Omega_r(z)}{\osc} w 
\leq C \Big(\frac{r}{ r_0}\Big)^\beta \, \|w\|_{L^\infty(\Omega_{r_0}(z))},
\end{align*} 
where $C>0$ depends only on $p$, $n$, 
 and $\Lambda$, while $\beta\in (0,1)$ depends  in addition on $c_*$.
\end{enumerate}
\end{theorem}
\begin{proof}
 The result in $(i)$ is from \cite[Theorem~1.1]{KK} and the precise estimate for $\nabla w$ 
 can be tracked from their proof and the use of Gehring's lemma. On the other hand, the result in $(ii)$ is 
 a particular case of Theorem~4.19  in \cite{MZ}. The  estimate obtained in \cite{MZ} is 
 less explicit than ours
 due to their more general equation. However, our stated oscillation estimate  follows from their proof in pages 200-201
 and our simple structural condition \eqref{simple-structural}.
\end{proof}

The next result is a special case of \cite[Lemma~5]{Li} and plays an essential role in proving our  main results.
\begin{theorem}[Boundary Lipschitz estimate]\label{thm:bnd-lipschitz} 
Let $\ba: \R^n \to \R^n$  be a continuous vector field such that  $\xi \mapsto\ba(\xi)$ is differentiable on $ \mathbb{R}^n\setminus \{0\}$ and  $\ba$ satisfies 
\eqref{structural-reference-1}--\eqref{structural-reference-2}. Suppose that   $w\in W^{1,p}(B_R^+)$ is  a weak solution
of $\div \ba(\nabla w)=0$ in $B^+_R$ and $w=0$ on $B_R \cap\{x: x_n =0\}$. Then we have
\begin{equation*}
 \sup_{B^+_{\frac{R}{3}}} |\nabla w|^p \leq C(p,n,\Lambda) \, \frac{1}{R^n}\int_{B^+_R} |\nabla w|^p dx.
 \end{equation*}
\end{theorem}

\subsection{Key regularity results}\label{sub:key-reg} 
\begin{theorem}\label{thm:main}
 Let  $\A$ satisfy \eqref{structural-reference-1}--\eqref{structural-reference-3} with $p>1$, and let $w\in A_s$ for some 
 $1<s<\infty$. 
For any $q\geq p$ and $M>0$,  there exists 
a  constant  $\delta=\delta( p, q,  n, \omega, \Lambda,  M, s, [w]_{A_s})>0$    such that:   if $\Omega$ is  $(\delta,R)$-Reifenberg flat,  $\lambda>0$, $\theta>0$, 
\eqref{smallness-1} holds, 
and  $u$ is a weak solution of \eqref{GE} satisfying 
 $\|u\|_{L^\infty(\Omega)}+\|\psi\|_{L^\infty(\Omega)}\leq \frac{M}{\lambda \theta}$,  then
 \begin{equation}\label{main-estimate}
\fint_{\Omega} |\nabla u|^q\, dw \leq  C\left( \|\nabla u\|_{L^p(\Omega)}^q + \fint_{\Omega}\M_{\Omega}(|\nabla\psi|^p
+|\F|^{p'})^{\frac{q}{p}} \, dw  \right).
\end{equation}
Here  $C>0$ is a constant 
 depending only on  $q$, $p$, $n$, $\omega$,  $\Lambda$,   $M$,  $s$, $R$, $\diam(\Omega)$, and $[w]_{A_s}$.
\end{theorem}
We also have the following localized version of Theorem~\ref{thm:main}.
\begin{theorem}\label{thm:localized-main}
 Let  $\A$ satisfy \eqref{structural-reference-1}--\eqref{structural-reference-3} with $p>1$, and let $w\in A_s$ for some 
 $1<s<\infty$. 
For any $q\geq p$ and $M>0$,  there exists 
a  constant  $\delta=\delta( p, q,  n, \omega, \Lambda,  M, s, [w]_{A_s})>0$    such that:  
if $\Omega$ is  $(\delta,R)$-Reifenberg flat,  $\lambda>0$, $\theta>0$, \eqref{smallness-1} holds,
and  $u$ is a weak solution of \eqref{GE} satisfying 
 $\|u\|_{L^\infty(\Omega)}+\|\psi\|_{L^\infty(\Omega)}\leq \frac{M}{\lambda \theta}$,  then
 \begin{align*}
 \frac{1}{w (B_{r}( y))} \int_{\Omega_{r}( y)}|\nabla u|^q d w 
 \leq  C\Bigg[ \Big(\frac{1}{|B_{2 r}( y)|}\int_{\Omega_{2r}( y)}|\nabla u|^p\, dx\Big)^{\frac{q}{p}}+
\frac{1}{ w (B_{r}( y))} 
\int_{\Omega_{r}(y)}\M_{\Omega_{2 r}( y)}(
|\nabla\psi|^p +| \F|^{p'})^{\frac{q}{p}} \, d w  \Bigg]
\end{align*}
for every $y\in \overline\Omega$ and  $r>0$.
Here  $C>0$  depends only on  $q$, $p$, $n$, $\omega$,  $\Lambda$,   $M$,  $s$, $R$, $\diam(\Omega)$, and $[w]_{A_s}$.
\end{theorem}
The  above two results play a crucial role in proving our main theorems stated in
Section~\ref{sec:Intro}, and their proofs will be  given in Subsection~\ref{sub:Lebesgue-Spaces}. 
As a consequence of Theorem~\ref{thm:main} and  the maximal function estimate,  we get:

\begin{corollary}[global weighted $L^q$ estimate]\label{cor:weighted} Let  $\A$ satisfy
\eqref{structural-reference-1}--\eqref{structural-reference-3} with $p>1$.
Then for any $q>p$, $M>0$, and any weight $w\in A_{\frac{q}{p}}$, there exists a  constant  $\delta>0$  such that:  if $\Omega$ is  $(\delta,R)$-Reifenberg flat, 
 $\lambda>0$, $\theta>0$, \eqref{smallness-1} holds, 
and  $u$ is a weak solution of \eqref{GE} satisfying 
 $\|u\|_{L^\infty(\Omega)}+\|\psi\|_{L^\infty(\Omega)}\leq \frac{M}{\lambda \theta}$,  we have
\begin{equation*}
\fint_{\Omega} |\nabla u|^q\, dw \leq C \fint_{\Omega}\Big(|\nabla \psi|^q + |\F|^{\frac{q}{p-1}}\Big)\, dw.
\end{equation*}
Here  $C,\, \delta$ are constants 
 depending only on  $q$, $p$, $n$, $\omega$, $\Lambda$,   $M$,  $ R$, $\diam(\Omega)$, and $[w]_{A_{\frac{q}{p}}}$.
\end{corollary}
\begin{proof}
Since $q>p$, we have from Theorem~\ref{thm:main} and Muckenhoup's strong type weighted estimate for the maximal function that
\begin{equation}\label{nablau-F}
\fint_{\Omega} |\nabla u|^q\, dw \leq C\left(\|\nabla u\|_{L^p(\Omega)}^q  +
 \fint_{\Omega} \big[|\nabla\psi|^q + |\F|^{\frac{q}{p-1}}\big]\, dw\right).
\end{equation}
From Proposition~\ref{prop:energy}
and  H\"older inequality, we also have
\begin{align}\label{nabla-varphi-F}
 \|\nabla u\|_{L^p(\Omega)}^q\leq  C \Big[\int_{\Omega} \big( |\nabla\psi|^p +|\F|^{p'}\big) \, dx\Big]^{\frac{q}{p}} 
 \leq  C  \Big[\int_{\Omega} w^{\frac{-p}{q-p}} dx\Big]^{\frac{q-p}{p}}
  \int_{\Omega} \big( |\nabla\psi|^q + |\F|^{\frac{q}{p-1}}\big)\, dw.
\end{align} 
Let $\bar x\in \R^n$ be such that $\Omega\subset B_0 :=B(\bar x, \diam(\Omega))$. Then 
\[
 \Big[\int_{\Omega} w^{\frac{-p}{q-p}} dx\Big]^{\frac{q-p}{p}} 
 \leq  \Big[\int_{B_0} w^{\frac{-p}{q-p}} dx\Big]^{\frac{q-p}{p}}\leq [w]_{A_{\frac{q}{p}}} \frac{|B_0|^{\frac{q}{p}}}{w(B_0)} 
 \leq [w]_{A_{\frac{q}{p}}} \frac{|B_0|^{\frac{q}{p}}}{w(\Omega)}.
\]
 This together with  \eqref{nablau-F}--\eqref{nabla-varphi-F} yields the desired conclusion.
\end{proof}

\section{Approximating gradients of solutions }\label{approximation-gradient}
In this section, we always suppose   that $\Omega$ is a bounded domain in $\R^n$ with $n\geq 2$. For the next result, we also assume that  
\begin{equation}\label{uniform-density-cond}
0\in \partial\Omega\quad \mbox{and}\quad \big| B_\rho(z) \setminus 
\Omega \big| \geq \frac13 \, |B_\rho(z)|\quad \mbox{for all $z\in B_4\cap \partial
\Omega\,$  and all $\, 0<\rho<2$}.
\end{equation}
\begin{lemma}\label{lm:local-boundary} Let $\A$ satisfy \eqref{structural-reference-1}--\eqref{structural-reference-3}, and 
$M>0$.
For any $\e\in (0,1]$, there exist small positive constants $\delta$ and $\sigma$  depending only on $\e$,   $p$,    $n$, $\omega$, $\Lambda$, and $M$ such that: 
if  $ \lambda>0$, $\theta>0$, $\Omega$ satisfies \eqref{uniform-density-cond},
\begin{align}\label{boundary-cond-at-zero}
 B_{3\sigma} \cap \{x_n> 3\sigma\delta\}  \subset \Omega_{3\sigma} \subset B_{3\sigma}\cap \{x_n> -3\sigma\delta\},
\end{align}
$y\in B_1$ satisfies either $y=0$ or $B_{4\sigma}(y) \subset \Omega_2$,
\begin{align}
\Theta_{\Omega_{3 \sigma }(y)}(\A) \leq \delta \quad \mbox{and}\quad \frac{1}{|B_4|}\int_{\Omega_{4}}  \big(|\nabla \psi|^p 
+|\F|^{p'}\big) \, dx  \leq \delta,
\nonumber
\end{align}
and $u$ is a weak solution of \eqref{GE}
 satisfying
\begin{equation*}
 \|u\|_{L^\infty(\Omega_{4})} +\|\psi\|_{L^\infty(\Omega_{4})}\leq \frac{M}{\lambda\theta},\quad 
  \frac{1}{|B_4|}\int_{\Omega_{4 }}{|\nabla  u|^p \, dx}\leq 1,
 \quad\mbox{and}\quad
  \frac{1}{|B_{4\sigma}(y)|}\int_{\Omega_{4\sigma}(y)}{|\nabla  u|^p dx}\leq 1,
\end{equation*}
then there exists a function  $v\in  W^{1,p}(\Omega_{2 \sigma }(y))$ such that
\begin{align*}
\|\nabla  v\|_{L^\infty(\Omega_{2 \sigma}(y))}
\leq C(p,n, \Lambda)
\quad\mbox{and}\quad 
\frac{1}{|B_{2\sigma}(y)|} \int_{\Omega_{2\sigma }(y)}{|\nabla  u - \nabla  v|^p\, dx}\leq \e^p.
\end{align*}
\end{lemma}
By translating and scaling, we obtain: 
\begin{proposition}
\label{lm:localized-compare-gradient}  Let $\A$ satisfy \eqref{structural-reference-1}--\eqref{structural-reference-3}, and 
$M>0$.
For any $\e\in (0,1]$, there exist small positive constants $\delta$ and $\sigma$  depending only on $\e$,   $p$,    $n$, 
$\omega$, $\Lambda$, and $M$ such that: 
if  $\Omega$ is  $(\delta,R)$-Reifenberg flat,  $ \lambda>0$, $\theta>0$, $\bar y\in \partial\Omega$,   $r\in (0,\frac{R}{2})$,
$y\in B_r(\bar y)$ satisfies either $y=\bar y$ or $B_{4\sigma r}(y) \subset \Omega_{2 r}(\bar y)$,
\begin{align}
\Theta_{\Omega_{3 \sigma r }( y)}(\A) \leq \delta \quad \mbox{and}\quad \frac{1}{|B_{4 r}(\bar y)|}\int_{\Omega_{4 r}(\bar y)}  
 \big(|\nabla \psi|^p| +\F|^{p'}\big) \, dx  
\leq \delta,\nonumber
\end{align}
and $u$ is a weak solution of \eqref{GE}
  satisfying
\begin{equation*}
 \|u\|_{L^\infty(\Omega_{4 r}(\bar y))} +\|\psi\|_{L^\infty(\Omega_{4 r}(\bar y))}\leq \frac{M}{\lambda\theta},
 \quad  \frac{1}{|B_{4 r}(\bar y)|}\int_{\Omega_{4 r}(\bar y)}{|\nabla  u|^p \, dx}\leq 1,
 \quad\mbox{and}\quad
  \frac{1}{|B_{4\sigma r}( y)|}\int_{\Omega_{4\sigma r}( y)}{|\nabla  u|^p dx}\leq 1,
\end{equation*}
then there exists a function  $v\in  W^{1,p}(\Omega_{2 \sigma r }( y))$ such that
\begin{align*}
\|\nabla  v\|_{L^\infty(\Omega_{2 \sigma r}( y))}
\leq C(p,n, \Lambda)
\quad\mbox{and}\quad 
\frac{1}{|B_{2\sigma r}( y)|} \int_{\Omega_{2\sigma r }( y)}{|\nabla  u - \nabla  v|^p\, dx}\leq \e^p.
\end{align*}
 \end{proposition}
\begin{proof}
The result is obtained  by translating and scaling, and then  applying
Lemma~\ref{lm:local-boundary}. Precisely, let  $\tilde\Omega := \{r^{-1} (x -\bar y): \, x\in\Omega\}$ and 
$\tilde y :=r^{-1} (y -\bar y)\in B_1$. Then
$0\in \partial\tilde\Omega$ and $\tilde \Omega$ is  $(\delta,R/r)$-Reifenberg flat
with  $R/r>2$. Thus it follows from Definition~\ref{def:Re} that $\tilde \Omega$ 
satisfies the  so-called (A)-property: 
there exists a positive constant $K=K(\delta)$  such that
\begin{equation*}\label{A-property}
 K |B_\rho(z)|\leq  |B_\rho(z)\cap \tilde\Omega|\leq (1-K) |B_\rho(z)|
\end{equation*}
for all $z\in \partial\tilde\Omega$ and all $\rho\in (0,R/r)$. Moreover, $K(\delta) \to 1/2$ when $\delta\to 0^+$. 
As  $\delta>0$ is small, we deduce  that   $\tilde \Omega$ satisfies condition \eqref{uniform-density-cond}. By 
rotating  the  standard coordinate system $\{x_1,...,x_n\}$ if necessary,  we also have from Definition~\ref{def:Re} that
\begin{equation*}
 B_{ \rho} \cap \{x_n>  \rho\delta\}  \subset \tilde\Omega_{\rho} 
\subset B_{\rho}\cap \{x_n> -\rho\delta\}
\end{equation*}
for any $\rho\leq 2$.  In particular, condition \eqref{boundary-cond-at-zero} is verified for $\tilde\Omega$ as well. 
We next define
\[
\tilde\A(x, z, \xi) = \A(r x +\bar y, z,\xi),\quad \tilde\F(x)=\F(r x +\bar y), \quad \tilde u(x) = r^{-1} u(r x +\bar y),
\quad \tilde\psi(x) = r^{-1} \psi(r x +\bar y), \quad\mbox{and }\, \,  \tilde\theta = \theta r. 
\]
Then $\tilde u$ is a weak solution of  $\div \Big[\frac{\tilde\A(x,\lambda\tilde\theta \tilde u, \lambda \nabla \tilde u)}{\lambda^{p-1}}\Big] = \div \tilde \F\,$ in 
$\tilde \Omega$ and $\tilde u=\tilde \psi$ on $\partial\tilde \Omega$. Moreover, 
\begin{align*}
 &\|\tilde u\|_{L^\infty(\tilde\Omega)} +\|\tilde\psi\|_{L^\infty(\tilde\Omega)}\leq \frac{M}{\lambda\tilde\theta}, \quad \frac{1}{|B_{4}|}
  \int_{\tilde\Omega_{4}}{|\nabla \tilde u|^p \, dx}= \frac{1}{|B_{4 r}(\bar y)|}
  \int_{\Omega_{4 r}(\bar y)}{|\nabla  u|^p \, dz}\leq 1,\\
  &\frac{1}{|B_{4\sigma}(\tilde y)|}
  \int_{\tilde\Omega_{4\sigma}(\tilde y)}{|\nabla  \tilde u|^p \, dx}= \frac{1}{|B_{4\sigma r}( y)|}
  \int_{\Omega_{4 \sigma r}( y)}{|\nabla  u|^p \, dz}\leq 1,
  \quad \Theta_{\tilde\Omega_{3\sigma}(\tilde y)}(\tilde \A)=\Theta_{\Omega_{3\sigma r}( y)}(\A)\leq \delta,\\
  &\mbox{and}\quad 
  \frac{1}{|B_{4}|}  \int_{\tilde\Omega_{4}} \big( |\nabla \tilde\psi|^p + |\tilde \F|^{p'}\big) \, dx
 =\frac{1}{|B_{4 r}(\bar y)|}  \int_{\Omega_{4 r}(\bar y)} \big( |\nabla \psi|^p +|\F|^{p'}\big) \, dz \leq \delta.
\end{align*}
Therefore, we can apply Lemma~\ref{lm:local-boundary} to conclude  that
there exists a function  $\tilde v\in  W^{1,p}(\tilde\Omega_{2 \sigma }(\tilde y))$ such that
\begin{equation*}
\|\nabla  \tilde v\|_{L^\infty(\tilde\Omega_{2 \sigma}(\tilde y))}
\leq C(p,n, \Lambda)
\quad\mbox{and}\quad 
\frac{1}{|B_{2\sigma}(\tilde y)|} \int_{\tilde\Omega_{2\sigma }(\tilde y)}{|\nabla  \tilde u - \nabla \tilde  v|^p\, dx}\leq \e^p.
\end{equation*}
Let $v(x):= r \tilde v(r^{-1} (x-\bar y))$. Then  we infer  that
\begin{align*}
\|\nabla  v\|_{L^\infty(\Omega_{2 \sigma r}(y))}
\leq C(p,n, \Lambda)
\quad\mbox{and}\quad 
\frac{1}{|B_{2\sigma r}( y)|} \int_{\Omega_{2\sigma r }(y)}{|\nabla  u - \nabla  v|^p\, dx}\leq \e^p.
\end{align*}
\end{proof}

 The rest of this section is devoted to proving Lemma~\ref{lm:local-boundary}. The first step is: 
\begin{lemma}\label{lm:step1}
For any $\e>0$, there exist small positive constants  $\delta$ and $\sigma$ depending only on $\e$,  $p$,  
$n$, $\omega$, $\Lambda$,   and $M$ such that: 
if  $ \lambda>0$, $\theta>0$, $\Omega$ satisfies \eqref{uniform-density-cond}, 
$y\in B_1$ satisfies either $y=0$ or $B_{4\sigma}(y) \subset \Omega_2$, 
$ \frac{1}{|B_4|}  \int_{\Omega_{4}}   \big(|\nabla \psi|^p 
+|\F|^{p'}\big) \, dx  \leq \delta$,
and $u$ is a weak solution of 
\eqref{GE} 
satisfying
\begin{equation*}
 \|u\|_{L^\infty(\Omega_{4})}+ \|\psi\|_{L^\infty(\Omega_{4})}\leq \frac{M}{\lambda\theta}
 \quad\mbox{ and }\quad \frac{1}{|B_4|} \int_{\Omega_{4}}{|\nabla  u|^p \, dx}\leq 1,
\end{equation*}
then 
\begin{equation*}
\frac{1}{|B_{4\sigma }(y)|} \int_{\Omega_{4\sigma }(y)}{|\nabla  u - \nabla  f|^p\, dx}\leq \e^p,
\end{equation*}
where $f$ is a weak solution of 
\begin{equation} \label{eq-f} 
 \left\{
 \begin{alignedat}{2}
  &\div \Big[\frac{\A(x,\lambda\theta \bar{u}_{\Omega_{4 \sigma}(y)}, \lambda \nabla f)}{\lambda^{p-1}}\Big]=0\quad \mbox{in}\quad \Omega_{ 4 \sigma}(y),\\\
&f=h \qquad\qquad\qquad\qquad \qquad\,\,\,\,\,\mbox{on}\quad \partial \Omega_{ 4 \sigma}(y)
 \end{alignedat} 
  \right.
\end{equation}
with 
$h$ being a weak solution of 
\begin{equation} \label{eq-h}
 \left\{
 \begin{alignedat}{2}
  &\div \Big[\frac{\A(x,\lambda\theta  u, \lambda \nabla h)}{\lambda^{p-1}}\Big]=0\quad \mbox{in}\quad \Omega_{4},\\\
&h= u-\psi \qquad\qquad\qquad\quad \mbox{on}\quad \partial \Omega_{4}.
 \end{alignedat} 
  \right.
\end{equation}
\end{lemma}
\begin{proof}
We only present the proof for   $p\geq 2$ using an idea in \cite{AM,BPS}. 
The argument for the case $1<p<2$ is similar with some 
slight adjustments which can be found in  \cite{BPS, DiN}. For convenience, let $\tilde \A(x, z,\xi) := \frac{\A(x,\lambda \theta z, \lambda \xi)}{\lambda^{p-1}}$.  We write
\[
\nabla u -\nabla f = \nabla (u-h) +\nabla (h-f)
\]
and will estimate $\| \nabla (u-h)\|_{L^p(\Omega_{4\sigma}(y))}$ and  $\| \nabla (h-f)\|_{L^p(\Omega_{4\sigma}(y))}$.
   By using $u-\psi-h$ as a test function in the equations for $u$ and $h$ we have
\begin{equation*}
\int_{\Omega_4}\langle \tilde\A(x,u,\nabla u)- \tilde \A(x,u,\nabla h), \nabla(u-\psi-h) \rangle \, dx
= \int_{\Omega_4}\langle \F, \nabla(u-\psi-h) \rangle \, dx
\end{equation*}
yielding
\begin{equation*}
\int_{\Omega_4}\langle \tilde\A(x,u,\nabla u)- \tilde \A(x,u,\nabla h), \nabla(u-h) \rangle \, dx
\leq \Lambda \int_{\Omega_4} |\nabla\psi|\big[ |\nabla u|^{p-1} + |\nabla h|^{p-1}\big] \, dx +
\int_{\Omega_4}|\F|\Big( |\nabla(u-h)| +|\nabla \psi|\Big) \, dx.
\end{equation*}
We then  use \eqref{structural-consequence} to bound the above left hand side from below. As a consequence, we obtain
\begin{equation*}
\int_{\Omega_4}| \nabla(u-h) |^p dx\leq C(p,\Lambda) \Bigg[
\int_{\Omega_4} |\nabla\psi|\big[ |\nabla u|^{p-1} + |\nabla h|^{p-1}\big] \, dx +
\int_{\Omega_4}|\F|\Big( |\nabla(u-h)| +|\nabla \psi|\Big)  \, dx\Bigg].
\end{equation*}
Hence  we infer from Young and H\"older  inequalities, and the energy estimate in Proposition~\ref{prop:energy}  that
\begin{align*}
 \int_{\Omega_4}| \nabla(u-h) |^p dx
 &\leq C \Bigg[ \|\nabla \psi\|_{L^p(\Omega_4)}
 \Big( \|\nabla u\|_{L^p(\Omega_4)}^{p-1} + \|\nabla h \|_{L^p(\Omega_4)}^{p-1} +\|\F\|_{L^{p'}(\Omega_4)}
 \Big) +  \int_{\Omega_4}  | \F|^{p'} dx \Bigg]\\
 &\leq C \Bigg[ \|\nabla \psi\|_{L^p(\Omega_4)}
 \Big( \|\nabla u\|_{L^p(\Omega_4)}^{p-1} + \|\nabla (u-\psi) \|_{L^p(\Omega_4)}^{p-1} +\|\F\|_{L^{p'}(\Omega_4)}
 \Big) +  \int_{\Omega_4}  | \F|^{p'} dx \Bigg]\\
 &\leq C \Bigg[ \|\nabla \psi\|_{L^p(\Omega_4)}
 \Big( \|\nabla u\|_{L^p(\Omega_4)}^{p-1} + \|\nabla \psi \|_{L^p(\Omega_4)}^{p-1} +\|\F\|_{L^{p'}(\Omega_4)}
 \Big) +  \int_{\Omega_4}  | \F|^{p'} dx \Bigg].
\end{align*}
Using the assumptions we then obtain
\begin{align}\label{u-h-0}
 \int_{\Omega_4}| \nabla(u-h) |^p dx\leq C \Big[
 \|\nabla \psi\|_{L^p(\Omega_4)} +
 \int_{\Omega_4}  | \F|^{p'} dx \Big],
\end{align}
which together with the fact that $B_{4\sigma}(y)\subset B_2$ implies that
\begin{equation}\label{u-h}\frac{1}{|B_{4\sigma}(y)|}
\int_{\Omega_{4 \sigma}(y)}| \nabla(u-h) |^p dx\leq \frac{C }{\sigma^n} \Big[
 \|\nabla \psi\|_{L^p(\Omega_4)} +
 \int_{\Omega_4}  | \F|^{p'} dx \Big].
\end{equation}
  By letting $m:=\bar u_{\Omega_{4 \sigma}(y)}$ and using $h-f$ as a test function in the equations for $h$ and $f$,  we have
\begin{equation*}
\int_{\Omega_{4 \sigma}(y)}\langle \tilde \A(x,m,\nabla f), \nabla(h-f) \rangle \, dx = \int_{\Omega_{4 \sigma}(y)}\langle \tilde \A(x,u,\nabla h), \nabla(h-f) \rangle \, dx.
\end{equation*}
This together with  \eqref{structural-consequence} gives
\begin{align}\label{est-nabla-h-f}
\int_{\Omega_{4 \sigma}(y)} |\nabla (h-f)|^p\, dx 
&\leq 4^{p-1}\Lambda \int_{\Omega_{4 \sigma}(y)}\langle \tilde\A(x,m,\nabla h)- \tilde \A(x,
m,\nabla f), \nabla(h-f) \rangle \, dx\nonumber\\
&= 4^{p-1}\Lambda \int_{\Omega_{4 \sigma}(y)}\langle \tilde\A(x,m,\nabla h)- \tilde \A(x,u,\nabla h), \nabla(h-f) \rangle \, dx\nonumber\\
&\leq  4^{p-1}\Lambda \int_{\Omega_{4 \sigma}(y)}\min{\big\{2\Lambda, \omega(\lambda \theta |u-
m|)\big\}}\,\, |\nabla h|^{p-1} |\nabla(h-f)| \, dx.
\end{align}
As a consequence of \eqref{est-nabla-h-f} and Young's inequality, we obtain
\[
\int_{\Omega_{4\sigma}(y)} |\nabla (h-f)|^p\, dx \leq C(p,\Lambda) \int_{\Omega_{4 \sigma}(y)} |\nabla h|^p \, dx.
\]
Let $\phi \in C_0^\infty(B_4)$ be the standard cutoff function satisfying $0\leq \phi\leq 1$, $\phi=1$ in $B_2$, and $|\nabla \phi|\leq C$.
Then it follows  by taking  $h \phi^p$ as a test function in  equation \eqref{eq-h} for $h$ that 
$\int_{\Omega_{2}} |\nabla h|^p \, dx \leq C  \int_{\Omega_{4}} |h|^p\, dx$.
Thus by combining with the above estimate and the fact $\Omega_{4 \sigma}(y)\subset \Omega_2$  we conclude that
\begin{equation*}
\frac{1}{|B_{4\sigma}(y)|} \int_{\Omega_{4\sigma}(y)} |\nabla (h-f)|^p\, dx \leq  \frac{C}{\sigma^n} \int_{\Omega_{4}} |h|^p\, dx\leq 
\frac{ C_*}{\sigma^n}   \|h\|_{L^\infty(\Omega_4)}^p\leq \frac{C_*}{\sigma^n}    \|u-\psi\|_{L^\infty(\Omega_4)}^p\leq \frac{C_*}{\sigma^n} \Big(\frac{M}{\lambda \theta}\Big)^p.
\end{equation*}
This  together with \eqref{u-h} gives  the desired conclusion if $ \frac{C_*}{\sigma^n}  \Big(\frac{M}{\lambda \theta }\Big)^p 
\leq \e^p/2$. We hence only need to consider the case \begin{equation}\label{small-para} 
\frac{C_*}{\sigma^n}    \Big(\frac{M}{\lambda \theta }\Big)^p >\frac{\e^p}{2}.
\end{equation}
 For this, note first that \eqref{u-h-0} and the assumption yield
\begin{equation*}
\|\nabla h\|_{L^p(\Omega_4)} \leq \|\nabla (h-u)\|_{L^p(\Omega_4)} +\|\nabla u\|_{L^p(\Omega_4)}\leq 
C\Bigg[\Big(\|\nabla \psi\|_{L^p(\Omega_4)}+ \int_{\Omega_4} |\F|^{p'}dx\Big)^{\frac1p}  +1 \Bigg]\leq C.
\end{equation*}
As  $h=0$ on $B_4 \cap \partial \Omega$ and $\Omega$ satisfies \eqref{uniform-density-cond},  
we can use  this estimate together with the higher integrability for $\nabla h$ given by Theorem~\ref{thm:inter-regularity} to conclude that
\begin{equation}\label{control-Lp0-Dh}
\Big(\frac{1}{|B_{4\sigma}(y)|} \int_{\Omega_{4\sigma}(y)}  |\nabla h|^{p_0} dx\Big)^{\frac{1}{p_0}} 
\leq C\,  \Big(\frac{1}{|B_{1}(y)|} \int_{\Omega_{1}(y)}  |\nabla h|^{p}  dx\Big)^{\frac{1}{p}}\leq C
\end{equation}
with $p_0>p$ and $C>0$ depending only on $p$, $n$, and $\Lambda$. 
We deduce from
 \eqref{est-nabla-h-f}, Young and H\"older inequalities, and \eqref{control-Lp0-Dh}  that 
\begin{align*}
\frac{1}{|B_{4\sigma}(y)|} \int_{\Omega_{4\sigma}(y)} |\nabla (h-f)|^p\, dx 
&\leq C\frac{1}{|B_{4\sigma}(y)|} \int_{\Omega_{4\sigma}(y)} \omega(\lambda \theta |u-m|)^{p'}\, |\nabla h|^p  dx\\
&\leq C
\Big[\frac{1}{|B_{4\sigma}(y)|} \int_{\Omega_{4 \sigma}(y)} \omega(\lambda \theta |u-m|)^{\frac{p' p_0}{p_0 - p}}  dx\Big]^{\frac{p_0 - p}{p_0}} \Big[\frac{1}{|B_{4\sigma}(y)|} \int_{\Omega_{4\sigma}(y)}  |\nabla h|^{p_0} dx\Big]^{\frac{p}{p_0}}\\
&\leq  C \Big[\frac{1}{|B_{4\sigma}(y)|} 
\int_{\Omega_{4 \sigma}(y)} \omega(\lambda \theta |u-m|)^{\frac{p' p_0}{p_0 - p}}  dx
\Big]^{\frac{p_0 - p}{p_0}}.
\end{align*}
But  for any $\gamma>0$, we have 
\begin{align*}
\int_{\Omega_{4 \sigma}(y)} \omega(\lambda \theta |u-m|)^{\frac{p' p_0}{p_0 - p}}  dx
&= \int_{\{\Omega_{4\sigma}(y): \lambda \theta |u-m| \leq \gamma  \}}
\omega(\lambda \theta |u-m|)^{\frac{p' p_0}{p_0 - p}}  dx +\int_{\{\Omega_{4 \sigma}(y): \lambda \theta |u-m| > \gamma  \}} 
\omega(\lambda \theta |u-m|)^{\frac{p' p_0}{p_0 - p}}  dx \\
&\leq |\Omega_{4\sigma}(y)| \, \omega(\gamma)^{\frac{p' p_0}{p_0 - p}}  +\frac{\omega(2 M)^{\frac{p' p_0}{p_0 - p}} }{\gamma^p} 
\int_{\Omega_{4 \sigma}(y)} \big(\lambda \theta |u-m|\big)^p dx.
\end{align*}
Therefore, we infer that
\begin{equation}\label{est-delta}
\frac{1}{|B_{4\sigma}(y)|} \int_{\Omega_{4 \sigma}(y)} |\nabla (h-f)|^p\, dx 
\leq C  \omega(\gamma)^{p'} + C \frac{\omega(2 M)^{p'}}{\gamma^{\frac{p(p_0 -  p)}{p_0}}}  
\Big[\frac{(\lambda \theta)^p}{|B_{4\sigma}(y)|}\int_{\Omega_{4 \sigma}(y)} 
|u-m|^p dx\Big]^{\frac{p_0 - p}{p_0}}
\end{equation}
for all $\gamma>0$. Let us now estimate the  last integral in \eqref{est-delta}. As
\begin{align*}
\|u-m\|_{L^p(\Omega_{4\sigma}(y) )}
&\leq  \|u-(h+\psi)\|_{L^p(\Omega_{4\sigma}(y) )}
+  \|(h+\psi)-\overline{(h+\psi)}_{\Omega_{4 \sigma}(y)}\|_{L^p(\Omega_{4\sigma}(y) )}
+  \|\overline{(h+\psi)}_{\Omega_{4 \sigma}(y)} -\bar u_{\Omega_{4 \sigma}(y)}\|_{L^p(\Omega_{4\sigma}(y) )}\\
&\leq 2  \|u-(h+\psi)\|_{L^p(\Omega_4)}
+  \|\psi-\bar\psi_{\Omega_{4 \sigma}(y)}\|_{L^p(\Omega_{4\sigma}(y) )}
+  \|h-\bar h_{\Omega_{4 \sigma}(y)}\|_{L^p(\Omega_{4\sigma}(y) )},
\end{align*}
it follows from  Sobolev and Poincar\'e inequalities that
\begin{align}\label{Sobolev}
\frac{1}{|B_{4\sigma}(y)|} 
\int_{\Omega_{4\sigma}(y) }|u-m|^p dx
 \leq C  \Big[ \sigma^{-n}\int_{\Omega_{4}} |\nabla (u- h-\psi)|^p dx 
+\frac{\sigma^p}{|B_{4\sigma}(y)|}
\int_{\Omega_{4 \sigma}(y)} |\nabla \psi|^p dx +
\big(\underset{\Omega_{ 4\sigma}(y)}{\osc} h \big)^p\Big]. 
\end{align}
We  now use the fact $h=0$ on $B_4 \cap \partial \Omega$ and $\Omega$ satisfies \eqref{uniform-density-cond}  to estimate the oscillation of $h$ when    $y=0$ or $B_{ 4\sigma}(y)\subset \Omega_2$. In the first
case, we can directly employ  Theorem~\ref{thm:inter-regularity} to get
\begin{equation}\label{b-case}
 \underset{\Omega_{ 4\sigma}(y)}{\osc} h 
 =\underset{\Omega_{ 4\sigma}}{\osc} \, h \leq C\big(\frac{4 \sigma}{2}\big)^{\beta }  \| h\|_{L^\infty(\Omega_2)}
 \leq C \sigma^{\beta }  \| h\|_{L^\infty(\Omega_4)}.
\end{equation}
In the second case, let $r:=\dist(y,\partial\Omega)=|y-x_0|\leq |y|<1$ for some $x_0\in \partial\Omega\cap B_2$. Then as
$B_r(y)\subset B_{2r}(x_0)$, we have  from
Theorem~\ref{classical-Holder} and  Theorem~\ref{thm:inter-regularity} that
\begin{equation}\label{semi-b-case}
\underset{\Omega_{ 4\sigma}(y)}{\osc} h=\underset{B_{ 4\sigma}(y)}{\osc} h
\leq C \Big(\frac{4\sigma }{r}\Big)^\beta \underset{B_{ r}(y)}{\osc} h
\leq C \Big(\frac{\sigma }{r}\Big)^\beta \underset{\Omega_{ 2r}(x_0)}{\osc} h
\leq C \Big(\frac{\sigma }{r}\Big)^\beta \Big(\frac{2r }{2}\Big)^\beta\|h\|_{L^\infty(\Omega_{2}(x_0))}
 \leq  C \sigma^\beta\|h\|_{L^\infty(\Omega_4)}.
\end{equation}
From \eqref{Sobolev}--\eqref{semi-b-case}, the fact $ \| h\|_{L^\infty(\Omega_4)}\leq M/\lambda\theta$, and \eqref{small-para}, we obtain
\begin{align*}
\frac{(\lambda \theta)^p}{|B_{4\sigma}(y)|}\int_{\Omega_{4 \sigma} (y)}|u-m|^p\, dx 
\leq C  \Bigg[\sigma^{-2 n}\big(\frac{M }{\e }\big)^p  \Big(\int_{\Omega_{4}} |\nabla (u- h)|^p\, dx  
+\int_{\Omega_{4}} |\nabla \psi|^p\, dx\Big)
+\big(M \sigma^{\beta } \big)^p \Bigg]. 
\end{align*}
Plugging this estimate into \eqref{est-delta} gives 
\begin{align*}
&\frac{1}{|B_{4\sigma}(y)|} \int_{\Omega_{4\sigma}(y)} |\nabla (h-f)|^p\, dx \\
&\leq  C  \omega(\gamma)^{p'}  +  C  M^{\frac{p(p_0 -  p)}{p_0}} \omega(2 M)^{p'} \left[\frac{\sigma^{-2n} }{(\gamma \e  )^p} 
\Big(\int_{\Omega_{4}} |\nabla (u- h)|^p\, dx +
\int_{\Omega_{4}} |\nabla\psi|^p\, dx\Big)
+\big(\frac{\sigma^{\beta }}{\gamma }\big)^p
 \right]^{\frac{p_0 -  p}{p_0}}.
\end{align*}
By combining this with \eqref{u-h} and using \eqref{u-h-0} we obtain 
\begin{align*}
\frac{1}{|B_{4\sigma}(y)|} \int_{\Omega_{4\sigma}(y)} |\nabla (u-f)|^p\, dx
&\leq \frac{C}{\sigma^n}\Big( \|\nabla \psi\|_{L^p(\Omega_4)} + \int_{\Omega_4}|\F|^{p'} dx \Big)+ C \omega(\gamma)^{p'} \\
&+ C M^{\frac{p(p_0 -  p)}{p_0}} \omega(2 M)^{p'} 
\left[\frac{\sigma^{-2 n} }{(\gamma \e  )^p } \Big( \|\nabla \psi\|_{L^p(\Omega_4)} + \int_{\Omega_4}|\F|^{p'} dx \Big) +  \big(\frac{\sigma^{\beta }}{\gamma }\big)^p \right]^{\frac{p_0 -  p}{p_0}}
\end{align*}
for every $\gamma>0$. From this, we get the desired conclusion by  choosing $\gamma$ small first,
then $\sigma$, and  $\delta$ last.
\end{proof}

Our second step is to show that the gradient of the solution $f$ to \eqref{eq-f} 
can be approximated by the gradient of a solution to a homogeneous equation with constant coefficient. Precisely, we have:
\begin{lemma}\label{lm:step2} 
Let $\e\in (0,1]$, and let  $\sigma$ be its corresponding constant given by  Lemma~\ref{lm:step1}.  
Let  $\Omega$, $\psi$, $\F$,  $u$, and $h$  be  as in Lemma~\ref{lm:step1}, and  
assume in addition that $\frac{1}{|B_{4 \sigma}(y)|} \int_{\Omega_{4\sigma}(y)} |\nabla u|^p\, dx\leq 1$. 
Suppose that  $f$ is a weak solution of \eqref{eq-f}  and 
$w$ is a weak solution of 
\begin{equation} \label{eq-w}
 \left\{
 \begin{alignedat}{2}
  &\div \Big[\frac{\langle\A\rangle_{\Omega_{3\sigma}(y)}(\lambda\theta \bar u_{\Omega_{4 \sigma}(y)}, \lambda \nabla w)}{\lambda^{p-1}}\Big]=0\quad \mbox{in}\quad \Omega_{3\sigma}(y),\\\
&w=f \qquad\qquad \qquad\qquad\qquad\qquad\,\, \, \mbox{on}\quad \partial \Omega_{3\sigma}(y).
 \end{alignedat} 
  \right.
\end{equation}
There exist constants $p_0\in (p,\infty)$ and $C>0$ depending only on $p$, $n$, and $\Lambda$ such that: 
 if $p\geq 2$, then
\begin{equation}\label{est-oscillation-n}
\frac{1}{|B_{3 \sigma}(y) | }\int_{\Omega_{3 \sigma}(y)}{|\nabla  f- \nabla  w|^p\, dx}\leq C \,\,  \Theta_{\Omega_{3 \sigma}(y)}(\A)^{\frac{p_0 - p}{p_0}},
\end{equation}
and if  $1<p<2$, then
\begin{align*}
\frac{1}{|B_{3 \sigma}(y) | } \int_{\Omega_{3\sigma}(y)} |\nabla f-\nabla w|^p\, dx 
&\leq 
  \tau 
+ C   \tau^{(1-\frac2p)p'}\Theta_{\Omega_{3\sigma}(y)}(\A)^{\frac{p_0 - p}{p_0}}\quad \mbox{for every}\quad \tau\in (0,1).
\end{align*}
\end{lemma}
\begin{proof}  For convenience, define 
 $\ba(x, \xi) := \frac{\A(x,\lambda \theta m, \lambda \xi)}{\lambda^{p-1}}$ with $m:=\bar u_{\Omega_{4 \sigma}(y)}$.
 We first consider the case $p\geq 2$. Then    using \eqref{structural-consequence} we get 
\begin{equation}\label{p-geq-2-n}
\int_{\Omega_{3\sigma}(y)} |\nabla ( f- w)|^p\, dx 
\leq \Lambda \int_{\Omega_{3\sigma}(y)}\langle \langle\ba\rangle_{\Omega_{3\sigma}(y)}(\nabla f)- \langle\ba\rangle_{\Omega_{3\sigma}(y)}(\nabla w), \nabla (f- w) \rangle \, dx :=I.
\end{equation}
To estimate the term $I$, we use  $f-w$ as a test function in equations  \eqref{eq-w}  and \eqref{eq-f}  to obtain
\begin{equation*}
\int_{\Omega_{3\sigma}(y)}\langle \langle\ba\rangle_{\Omega_{3\sigma}(y)}(\nabla w), \nabla(f-w) \rangle \, dx= \int_{\Omega_{3\sigma}(y)}\langle \ba(x,\nabla f), \nabla(f-w) \rangle \, dx.
\end{equation*}
Hence,
\begin{align}\label{est-I-n}
I&= \Lambda \int_{\Omega_{3\sigma}(y)}\langle \langle\ba\rangle_{\Omega_{3\sigma}(y)}(\nabla f)-  \ba(x,\nabla f), \nabla(f-w) 
\rangle \, dx\nonumber\\
&\leq  \Lambda \int_{\Omega_{3\sigma}(y)} \sup_{\xi\neq 0} \frac{|\ba(x,\xi) -
 \langle\ba\rangle_{\Omega_{3\sigma}(y)}(\xi)| }{|\xi|^{p-1}}\,\, |\nabla f|^{p-1} |\nabla(f-w)| \, dx.
\end{align}
We now claim that there exist constants $p_0\in (p,\infty)$ and $C>0$ depending only on $p$, $n$, and $\Lambda$ such that
\begin{equation}\label{control-higher-Df}
\left(\frac{1}{|B_{3 \sigma}(y) | } \int_{\Omega_{3\sigma}(y)}  |\nabla f|^{p_0} dx\right)^{\frac{1}{p_0}}
\leq 
C\, \left(\frac{1}{|B_{4 \sigma} (y)| } \int_{\Omega_{4\sigma}(y)}  |\nabla f|^{p}  dx\right)^{\frac{1}{p}}.
\end{equation}
Indeed, this follows from  the classical interior higher integrability if $B_{ 4\sigma}(y)\subset \Omega_2$. In the case $y=0$, we  obtain \eqref{control-higher-Df}  from the boundary   higher integrability in  Theorem~\ref{thm:inter-regularity} by  using the fact 
$f=h=0$ on $B_{4\sigma} \cap \partial \Omega$ and  the assumption that $\Omega$ satisfies \eqref{uniform-density-cond}.    Thanks to Lemma~\ref{lm:step1}, we also have 
\begin{equation}\label{est-nabla-f-n}
\frac{1}{|B_{4 \sigma}(y) | } \int_{\Omega_{4\sigma}(y)}  |\nabla f|^{p}  dx 
\leq \frac{2^{p-1}}{|B_{4 \sigma}(y) | }  \left[ \int_{\Omega_{4\sigma}(y)} |\nabla f-\nabla u|^p\, dx  +
 \int_{\Omega_{4\sigma}(y)} |\nabla u|^p\, dx \right]\leq 2^{p-1}[\e^p +1 ]\leq 2^p.
\end{equation}
Therefore, we infer that 
\begin{equation*}
\left(\frac{1}{|B_{3 \sigma}(y) | } \int_{\Omega_{3\sigma}(y)}  |\nabla f|^{p_0} dx\right)^{\frac{1}{p_0}}
\leq 
C.
\end{equation*}
This together with   \eqref{p-geq-2-n}--\eqref{est-I-n},  Young and H\"older inequalities,  and the fact $|\ba(x,\xi)|\leq \Lambda |\xi|^{p-1}$ gives
\begin{align*}
&\frac{1}{|B_{3 \sigma}(y) | }\int_{\Omega_{3\sigma}(y)} |\nabla (f-w)|^p\, dx 
\leq C\frac{1}{|B_{3 \sigma}(y) | }  \int_{\Omega_{3\sigma}(y)}\Big[ \sup_{\xi\neq 0} \frac{|\ba(x,\xi) -
\langle\ba\rangle_{\Omega_{3\sigma}(y)}(\xi)| }{|\xi|^{p-1}}\Big]^{p'}\, |\nabla f|^p  dx\\
&\leq C
\left(\frac{1}{|B_{3 \sigma} (y) | } \int_{\Omega_{3\sigma}(y)}  \Big[ \sup_{\xi\neq 0} \frac{|\ba(x,\xi) -\langle\ba\rangle_{\Omega_{3\sigma}(y)}(\xi)| }{|\xi|^{p-1}}\Big]^{\frac{p' p_0}{p_0 - p}}  dx\right)^{\frac{p_0 - p}{p_0}} \left(\frac{1}{|B_{3 \sigma}(y) | } \int_{\Omega_{3\sigma}(y)}  |\nabla f|^{p_0} dx\right)^{\frac{p}{p_0}}\\
&\leq 
C \left(\frac{1}{|B_{3 \sigma} (y)| } \int_{\Omega_{3\sigma}(y)}  \sup_{\xi\neq 0} \frac{|\ba(x,\xi) -\langle\ba\rangle_{\Omega_{3\sigma}(y)}(\xi)| }{|\xi|^{p-1}} dx\right)^{\frac{p_0 - p}{p_0}}.
\end{align*}
But  we have from the definition of $\ba$ that
\begin{align*}
\sup_{\xi\neq 0} \frac{|\ba(x,\xi) -\langle\ba\rangle_{\Omega_{3\sigma}(y)}(\xi)| }{|\xi|^{p-1}} 
&= \sup_{\eta\neq 0} \frac{|\A(x,\lambda\theta m, \eta) 
-\langle\A\rangle_{\Omega_{3\sigma}(y)}(\lambda\theta m,
\eta)| }{|\eta|^{p-1}}.
 \end{align*}
Thus we conclude that
\begin{align*}
\frac{1}{|B_{3 \sigma}(y) | } \int_{\Omega_{3\sigma}(y)} |\nabla (f-w)|^p\, dx 
\leq 
C \left(\frac{1}{|B_{3 \sigma} (y)| } \int_{\Omega_{3\sigma}(y)}  \sup_{\xi\neq 0} \frac{|\A(x,\lambda 
\theta m, \xi) -\langle\A\rangle_{\Omega_{3\sigma}(y)}(\lambda 
\theta m, \xi)| }{|\xi|^{p-1}} dx\right)^{\frac{p_0 - p}{p_0}}.
\end{align*}
This together with 
the fact $\lambda\theta m\in \overline{\K} \cap [-M, M]$  and the definition of $\Theta_{\Omega_{3\sigma}}(\A)$ given by 
\eqref{def:Theta}
yields estimate \eqref{est-oscillation-n}.
We next consider 
the case $1<p<2$. Then   condition (1.2) in \cite{NP} is satisfied thanks to \eqref{structural-consequence}.
Therefore, instead of \eqref{p-geq-2-n}  we now 
 have from  \cite[Lemma~3.1]{NP} that 
\begin{equation*}
\int_{\Omega_{3\sigma}(y)} |\nabla (f-w)|^p\, dx 
\leq  \tau \int_{\Omega_{3\sigma}(y)} |\nabla f|^p\, dx + 
C_p  \tau^{1-\frac2p} I \quad\mbox{for all}\quad \tau\in (0,\frac12). 
\end{equation*}
Then we deduce from  estimate \eqref{est-I-n} for $I$ and Young's inequality that 
\begin{align*}
\frac{1}{|B_{3 \sigma}(y) | } \int_{\Omega_{3\sigma}(y)} |\nabla (f-w)|^p\, dx 
&\leq 
 2 \tau \frac{1}{|B_{3 \sigma}(y) | } \int_{\Omega_{3\sigma}(y)} |\nabla f|^p\, dx\\
&+ C(p,\Lambda)   \tau^{(1-\frac2p)p'}\frac{1}{|B_{3 \sigma}(y) | } \int_{\Omega_{3\sigma}(y)}\Big[ 
\sup_{\xi\neq 0} \frac{|\ba(x,\xi) -\langle\ba\rangle_{\Omega_{3\sigma}(y)}(\xi)| }{|\xi|^{p-1}}\Big]^{p'}\, |\nabla f|^p  dx.
\end{align*}
The first integral is estimated by \eqref{est-nabla-f-n}
 and  the last integral can be estimated exactly as above. As a consequence, we obtain 
\begin{align*}
\frac{1}{|B_{3 \sigma}(y) | } \int_{\Omega_{3\sigma}(y)} |\nabla (f-w)|^p\, dx 
&\leq 
 2^{p+1} \big(\frac43\big)^n \,  \tau 
+ C(p,n,\Lambda)   \tau^{(1-\frac2p)p'}\Theta_{\Omega_{3\sigma}(y)}(\A)^{\frac{p_0 - p}{p_0}}\quad\mbox{for all}\quad \tau\in (0,\frac12). 
\end{align*}
\end{proof}

To obtain Lemma~\ref{lm:local-boundary}, our last step is to show that the gradient of the solution $w$ to \eqref{eq-w} 
can be approximated by a bounded gradient. Precisely, we have:
\begin{lemma}\label{lm:step3}
Let $\ba:\R^n \to \R^n$ be a vector field as in Theorem~\ref{thm:bnd-lipschitz}.  Let  $\e>0$, and let  $\sigma$ be its corresponding constant given by 
Lemma~\ref{lm:step1}.  Then there exists a constant $\delta>0$ depending only on $\e$, $p$, $\Lambda$,  and  $n$ satisfying: 
if $\Omega\subset\R^n$ is  a bounded domain such that 
\eqref{boundary-cond-at-zero}
holds,  $y\in B_1$ satisfies either $y=0$ or $B_{4\sigma}(y) \subset \Omega$, and  
$w\in W^{1,p}(\Omega_{3\sigma}(y))$ is a weak solution of 
\begin{equation*} 
 \left\{
 \begin{alignedat}{2}
  &\div \, \ba(\nabla w) =0\quad \mbox{in}\quad \Omega_{3\sigma}(y),\\\
&w=0\qquad\qquad\,\, \mbox{on}\quad B_{3\sigma}(y) \cap  \partial \Omega
 \end{alignedat} 
  \right.
\end{equation*}
with  $ \frac{1}{|B_{3\sigma}(y)|}\int_{\Omega_{3\sigma}(y)} |\nabla w|^p\, dx\leq 1$, then there exists a function 
$v\in  W^{1,p}(\Omega_{2 \sigma }(y))$ such that
\begin{align*}\label{last-appro}
\|\nabla  v\|_{L^\infty(\Omega_{2 \sigma}(y))}
\leq C(p,n, \Lambda)
\quad\mbox{and}\quad 
\frac{1}{|B_{2\sigma}(y)|} \int_{\Omega_{2\sigma }(y)}{|\nabla  w - \nabla  v|^p\, dx}\leq \e^p.
\end{align*}
  \end{lemma}
\begin{proof} 
If $B_{4\sigma}(y)\subset \Omega$, then the conclusion follows by simply taking
$v=w$ and using the interior Lipschitz estimate for $w$. Notice that condition \eqref{boundary-cond-at-zero} is not used for this case.
Thus it suffices to consider the case $y=0$. If we let  $\tilde\Omega := \{\sigma^{-1} x : \, x\in\Omega\}$ and 
$ \tilde w(x) :=\sigma^{-1} w(\sigma x )$, then
$\tilde\Omega$ satisfies 
 \[B_{3}\cap \{x_n> 3 \delta\} \subset \tilde\Omega_{3} \subset B_{3}\cap \{x_n> -3\delta\}
  \]
  and $\tilde w$ is a weak solution of  $\div \ba(\nabla \tilde w) = 0$ in 
$\tilde \Omega_3$ with   $\tilde w=0$ on $B_3 \cap \partial\tilde \Omega$. Moreover,
$
  \frac{1}{|B_{3}|}
  \int_{\tilde\Omega_{3}}{|\nabla \tilde w|^p \, dx}= \frac{1}{|B_{3\sigma}|}
  \int_{\Omega_{3\sigma}}{|\nabla  w|^p \, dz}\leq 1$.
Therefore, by working with $\tilde\Omega$ and $\tilde w$ we can assume in addition that $\sigma=1$,
and only need to show that there exists a function 
$v\in  W^{1,p}(\Omega_{2})$ such that
\begin{align*}
\|\nabla  v\|_{L^\infty(\Omega_{2})}
\leq C(p,n, \Lambda)
\quad\mbox{and}\quad 
 \int_{\Omega_{2}}{|\nabla  w - \nabla  v|^p\, dx}\leq \e^p.
\end{align*}
 For this, we first apply  Lemma~\ref{lm:compactness} to conclude that there exists  a weak solution $\tilde v$ to the equation 
\begin{equation*} 
 \left\{
 \begin{alignedat}{2}
  &\div \,\ba(\nabla \tilde v)=0\quad \mbox{in}\quad B_{3}^+,\\\
&\tilde v=0 \qquad\qquad \mbox{ on}\quad B_{3} \cap \{x_n=0\}
 \end{alignedat} 
  \right.
\end{equation*}
with $ \frac{1}{|B_3|}\int_{B_3^+} |\nabla \tilde v|^p\, dx \leq (4^p\Lambda^2 )^p$  such that 
\begin{equation}\label{w-close-v-1}
\Big(\int_{B_{3} \cap \{x_n> 3\delta\} } |w-\tilde v|^p\, dx \Big)^{\frac1p}\leq \e^p.
\end{equation}
Then by the Lipschitz estimate in Theorem~\ref{thm:bnd-lipschitz}  we also have 
\begin{equation}\label{v-lipschitz-proof-first}
\|\nabla \tilde v\|_{L^\infty(B_\frac52^+)}^p
\leq C \int_{B_3^+} |\nabla \tilde v|^p\, dx\leq C.
\end{equation}
Let us extend $\tilde v$ from $B_{3}^+$ to $B_3$ by the zero extension. Due to \eqref{v-lipschitz-proof-first}
the resulting function $\tilde v$ satisfies: 
$\tilde v\in W^{1,p}(B_3)$ and
\begin{equation}\label{v-lipschitz-proof}
\|\nabla \tilde v\|_{L^\infty(B_\frac52)}
\leq  C.
\end{equation}
Moreover, $\tilde v$  is a weak solution of $\div \, \ba(\nabla \tilde  v)=-\tilde g_{x_n}$ in $B_3$
with 
$\tilde g(x) := \chi_{\{x_n <0 \} }(x) \, \ba_{n}(\nabla \tilde v(x', 0))$ for $x=(x', x_n)\in B_3$.  
Notice that if we let $D_\delta:= B_{3} \cap \{x_n> 3\delta\}$, then 
\eqref{w-close-v-1}  gives
\[
 \int_{\Omega_{\frac52}} |w-\tilde v|^p\, dx
\leq \int_{\Omega_{\frac52} \setminus D_\delta} |w-\tilde v|^p\, dx +\int_{D_\delta } |w-\tilde  v|^p\, dx 
\leq 2^{p-1}\Big[\int_{\Omega_{\frac52} \setminus D_\delta} |w|^p\, dx + \int_{\Omega_{\frac52} \setminus D_\delta} |\tilde v|^p\, 
dx\Big]
+\e^{p^2}.
\]
If $p>n$, then we can bound $L^\infty$ norms of $w$ and $\tilde v$ since
$\|w\|_{W^{1,p}(\Omega_\frac52)} +\|\tilde v\|_{W^{1,p}(\Omega_\frac52)}\leq C$. As a consequence, we deduce 
that 
\begin{align}\label{w-close-v-2}
\int_{\Omega_{\frac52}} |w-\tilde v|^p\, dx
\leq C
\big| \Omega_{\frac52} \setminus D_\delta\big|^{\frac{p}{n}} 
+\e^{p^2}\leq  C\delta^{\frac{p}{n}} +\e^{p^2}.
\end{align}
In case $p<n$, we also have  \eqref{w-close-v-2} as the above estimate together with   H\"older and Poincar\'e inequalities implies that
\begin{align*}
\int_{\Omega_{\frac52}} |w-\tilde v|^p\, dx
&\leq 2^{p-1}\Bigg[\Big(\int_{\Omega_{\frac52} \setminus D_\delta} |w|^{\frac{np}{n-p}}\, dx\Big)^{\frac{n-p}{n}}
+ \Big(\int_{B_{\frac52}^+ \setminus D_\delta} |\tilde v|^{\frac{np}{n-p}}\, dx\Big)^{\frac{n-p}{n}}\Bigg] 
\big| \Omega_{\frac52} \setminus D_\delta\big|^{\frac{p}{n}} 
+\e^{p^2}\nonumber\\
&\leq C\Bigg[\Big(\int_{\Omega_{\frac52} \setminus D_\delta} |\nabla w|^p\, dx\Big)^{\frac{n-p}{n}}
+ \Big(\int_{B_{\frac52}^+ \setminus D_\delta} |\nabla \tilde v|^p\, dx\Big)^{\frac{n-p}{n}}\Bigg] 
\delta^{\frac{p}{n}} 
+\eta^p\leq C\delta^{\frac{p}{n}} +\e^{p^2}.
\end{align*}
Let us define $v(x) := \tilde v(x, x_n -3\delta) =\tilde v( x-3\delta e_n)$. Then from the equation for $\tilde v$ we infer that
 $ v$  is a weak solution of 
\begin{equation} \label{eq-v-jump}
 \left\{
 \begin{alignedat}{2}
  &\div \, \ba(\nabla  v)=- g_{x_n}\quad \mbox{in}\quad \Omega_3,\\\
& v=0 \qquad\qquad \quad\,\,\, \mbox{ on}\quad B_{3} \cap \partial\Omega,
 \end{alignedat} 
  \right.
\end{equation}
where 
\[
g(x) := \chi_{\{x_n <3\delta \} }(x) \, \ba_{n}(\nabla  v(x', 3\delta))
=\chi_{\{x_n <3\delta \} }(x) \, \ba_{n}(\nabla  \tilde v(x', 0))\quad \mbox{for}\quad x=(x', x_n)\in \Omega_3.  
\]
Let $D^{3\delta}_n \tilde v(x) :=\frac{|\tilde v(x) -\tilde v(x-3\delta e_n)|}{3\delta}$ 
denote the $n$th difference quotient.
From \eqref{w-close-v-2} and $L^p$ estimate for $\nabla \tilde v$, we also have
\begin{align}\label{w-close-v-3}
\|w- v\|_{L^p(\Omega_{\frac52})} 
&\leq \|w-\tilde v\|_{L^p(\Omega_{\frac52})}  +\Big(\int_{\Omega_\frac52} |\tilde v(x) -\tilde v(x-3\delta e_n)|^p\, dx 
\Big)^{\frac1p}
\leq \|w-\tilde v\|_{L^p(\Omega_{\frac52})}  +3\delta \|D^{3\delta}_n \tilde v\|_{L^p(\Omega_{\frac52})}\nonumber\\
&\leq \|w-\tilde v\|_{L^p(\Omega_{\frac52})}  +3\delta \|\nabla \tilde v\|_{L^p(B_{\frac{11}{4}})}
\leq C\big(\delta^{\frac{1}{n}} +\delta\big) +\e^p
\leq C\delta^{\frac{1}{n}}
+\e^p .
\end{align}
Take $\phi\in C_0^\infty(B_{\frac52})$  be the standard cutoff function satisfying $\phi = 1$ in $B_2$. 
Then by using $\phi^p (w-v)$ as a test function in the equations for $w$ and \eqref{eq-v-jump}, we  obtain
\begin{align}\label{w-v-1}
\int_{\Omega_3} \langle \ba(\nabla w), \nabla \big[\phi^p (w-v)\big]  \rangle \, dx
=\int_{\Omega_3} \langle \ba(\nabla v), \nabla \big[\phi^p (w-v)\big]  \rangle \, dx + 
\int_{\Omega_3} g \big[\phi^p (w-v)\big]_{x_n} \, dx.
\end{align}
We can now follow the proof of \cite[Lemma~3.7]{BR} 
to get 
\[
\int_{\Omega_{2} } |\nabla (w-v)|^p\, dx \leq \e^p.
\]
For clarity, let us include the argument for the case $p\geq 2$. Indeed, we can rewrite \eqref{w-v-1} as
\begin{align*}
\int_{\Omega_3} \langle \ba(\nabla w)- \ba(\nabla v), \nabla (w-v)  \rangle \phi^p\, dx
&= p\int_{\Omega_3} \langle \ba(\nabla v)-\ba(\nabla w), \nabla \phi  \rangle \phi^{p-1} (w-v) \, dx\\
& \quad + 
\int_{\Omega_3\setminus D_\delta} \ba_{n}(\nabla  v(x', 3\delta)) \big[(w-v)_{x_n}\phi^p + p \phi_{x_n} \phi^{p-1}
(w-v)\big] \, dx.
\end{align*}
It follows that
\begin{align*}
\int_{\Omega_3} |\nabla (w-v)|^p   \phi^p\, dx
&\leq  C\int_{\Omega_3} \Big[ |\nabla v|^{p-1} + |\nabla w|^{p-1}\Big] |\nabla \phi| \phi^{p-1} |w-v| \, dx
+C \int_{\Omega_3\setminus D_\delta} |\nabla v(x', 3\delta)|^{p-1} |\nabla (w-v)| \phi^p  \, dx \\
&\quad +C \int_{\Omega_3\setminus D_\delta} |\nabla v(x', 3\delta)|^{p-1} \phi^{p-1}
|w-v| \, dx.
\end{align*}
Then Young and H\"older inequalities  yield
\begin{align*}
\int_{\Omega_3} |\nabla (w-v)|^p   \phi^p\, dx
&\leq C\Bigg[ \|\nabla v\|_{L^p(\Omega_3)}^{p-1} + \|\nabla w\|_{L^p(\Omega_3)}^{p-1} +\Big(\int_{\Omega_\frac52
\setminus D_\delta} |\nabla v(x', 3\delta)|^p
dx\Big)^{\frac{p-1}{p}}\Bigg] \|w-v\|_{L^p(\Omega_\frac52)} \\
&\quad +C \int_{\Omega_\frac52\setminus D_\delta} |\nabla v(x', 3\delta)|^p   \, dx. 
\end{align*}
This together with the fact $\nabla v(x', 3\delta) =\nabla \tilde v(x',0)$,  \eqref{v-lipschitz-proof}, and \eqref{w-close-v-3} 
gives
\[
\int_{\Omega_2} |\nabla (w-v)|^p   \, dx
\leq C\Bigg(\|w-v\|_{L^p(\Omega_\frac52)} +\big|\Omega_\frac52\setminus D_\delta\big|  \Bigg)
\leq C \Big( \delta^{\frac1n}  +\e^p+\delta\Big)\leq C ( \delta^{\frac1n}  +\e^p).
\]
\end{proof}

\begin{proof}[{\bf Proof of Lemma~\ref{lm:local-boundary}}]
 The conclusion follows from Lemmas~\ref{lm:step1}--\ref{lm:step3} noting that
 \[
\|\nabla u -\nabla v\|_{L^p(\Omega_{2\sigma}(y))}  
\leq \|\nabla u -\nabla f\|_{L^p(\Omega_{2\sigma}(y))} +\|\nabla f -\nabla w\|_{L^p(\Omega_{2\sigma}(y))} 
+\|\nabla w -\nabla v\|_{L^p(\Omega_{2\sigma}(y))}. 
 \]
\end{proof}

\section{Density and gradient estimates}\label{interior-density-gradient}
We derive global  gradient estimates for weak solution $u$ of \eqref{GE} by estimating the distribution functions of the maximal
function of $|\nabla  u|^p$. This is carried out in the next two subsections, while the last subsection (Subsection~\ref{sec:Morrey-Spaces}) 
is devoted to proving the main results stated in Section~\ref{sec:Intro}.
\subsection{Density estimates}\label{sub:density-est}
The next result gives a density estimate for the distribution  of $\M_{U}(|\nabla  u|^p)$. It roughly says that  if the maximal 
function $\M_{U}(|\nabla  u|^p)$ is bounded at one point in $B_{\sigma r}(y)$ then this property can be propagated for all points in $B_{\sigma r}(y)$ 
except on a set of small measure $w$. 

\begin{lemma}\label{A-initial-density-est}
Assume that   $\A$ satisfies \eqref{structural-reference-1}--\eqref{structural-reference-3}, $\psi\in W^{1,p}(\Omega)$, and  $\F\in L^{p'}(\Omega;\R^n)$. 
Let  $M>0$ and $w$ be an $A_\infty$ weight.  There exists a constant $N=N(p, n,\Lambda)>1$ satisfying  
for  any $\e>0$, we can find   
small positive constants $\delta$  and $\sigma$ depending only on $\e,\, p,\, n,\, \omega,\, \Lambda,\, M$, 
and $[w]_{A_\infty}$
such that:  if $\Omega$ is  $(\delta,R)$-Reifenberg flat,  
$\lambda>0$, $\theta>0$, $U\subset \Omega$ is an open set,  $\bar y\in \partial \Omega$, and
\begin{align*}
\sup_{0<\rho<\frac{R}{2}}\sup_{y\in \overline{\Omega}\cap B_{\frac{R}{2}}(\bar y)} \Theta_{\Omega_{3\sigma\rho}(y)}(\A)\leq  \delta,
\end{align*}
then
  for any weak solution $u$ of \eqref{GE} with
 $\|u\|_{L^\infty(\Omega)}+\|\psi\|_{L^\infty(\Omega)}\leq \frac{M}{\lambda \theta}$,
 for any $r\in (0, \frac{R}{2})$  satisfying  $\Omega_{4r}(\bar y)\subset U$,
 for  $y\in B_r(\bar y)$ satisfies either $y=\bar y$ or   $B_{4 \sigma r}(y) \subset \Omega_{2r}(\bar y)$, and
\begin{align}
 &B_{\sigma r}(y)\cap \big\{U:\, \M_{U}(|\nabla  u|^p)\leq 1 \big\}\cap
 \{ U: \M_{U}(|\nabla\psi|^p + |\F|^{p'} )\leq \delta\}\neq \emptyset,
\label{A-one-point-condition}
 \end{align}
we have
\begin{equation*}
 w\Big( \{ U:\, \M_{U}(|\nabla  u|^p)> N \}\cap B_{\sigma r}(y)\Big)
< \e  \, w(B_{\sigma r}(y)),
\end{equation*}
\end{lemma}
\begin{proof}
By \eqref{A-one-point-condition} there exists $x_0\in B_{\sigma r}(y)\cap U$ such that 
\begin{align}\label{A-max-one-point}
\M_{U}(|\nabla  u|^p)(x_0) \leq 1 \quad \mbox{and}\quad  \M_{U}(|\nabla \psi|^p + |\F|^{p'} )(x_0)\leq \delta.
\end{align}
This together with the facts $B_{4 r}(\bar y)\subset B_{6 r}(x_0)$ and $B_{4 \sigma r}(y)\subset B_{5 \sigma r}(x_0)\cap 
B_{2 r}(\bar y)$ implies that
\begin{align*}
&\frac{1}{|B_{4 r}(\bar y)|} \int_{\Omega_{4 r}(\bar y)}{|\nabla  u|^p dx}
\leq \big(\frac64\big)^n \frac{1}{|B_{6 r}(x_0)|} \int_{B_{6 r}(x_0)\cap U}{|\nabla  u|^p \, dx}\leq \big(\frac32\big)^n,\\
&\frac{1}{|B_{4 \sigma r}(y)|} \int_{\Omega_{4\sigma r}(y)}{|\nabla  u|^p dx}
\leq \big(\frac54\big)^n \frac{1}{|B_{5 \sigma r}(x_0)|} \int_{B_{5\sigma r}(x_0)\cap U}{|\nabla  u|^p \, dx}\leq \big(\frac54\big)^n,\\
  &\frac{1}{|B_{4 r}(\bar y)|} \int_{\Omega_{4r}(\bar y)} \big( |\nabla \psi|^p + |\F|^{p'} \big)\, dx
 \leq \big(\frac64\big)^n\frac{1}{|B_{6 r}(x_0)|} \int_{B_{6 r}(x_0)\cap U}{ \big(|\nabla \psi|^p + |\F|^{p'}\big)\,  dx}\leq 
 \big(\frac32\big)^n \delta.
\end{align*}
In addition, we have from the assumption that either $y=\bar y\in \partial \Omega $ or   $B_{4 \sigma r}(y) \subset \Omega_{2r}(\bar y) $. 
Therefore, we can apply Proposition~\ref{lm:localized-compare-gradient}   for 
$\tilde \e \in (0,1]$ that will be determined later.
As a consequence, we obtain there exists $v\in  W^{1,p}(\Omega_{2 \sigma r}(y))$ such that
\begin{equation}\label{A-eq:nabla-u-v}
\|\nabla  v\|_{L^\infty(\Omega_{2 \sigma  r}(y))}^p
\leq C_*(p,n,\Lambda) 
\quad \mbox{and}\quad 
\frac{1}{|B_{2\sigma r}(y)|}\int_{\Omega_{2\sigma r}(y)}{|\nabla  u - \nabla  v|^p\, dx}\leq \tilde\e^p.
\end{equation}
We claim that \eqref{A-max-one-point} and  \eqref{A-eq:nabla-u-v} yield
\begin{equation}\label{A-key-claim}
\big\{\Omega_{\sigma r}(y): \, \M_{\Omega_{2\sigma r}(y)}(|\nabla u - \nabla v|^p) \leq C_*\big\} 
\subset \big\{\Omega_{\sigma r}(y):\, \M_{U}(|\nabla  u|^p)\leq  N \big\}
\end{equation}
with $N := \max{\{2^{p} C_*, 3^n\}}$.
Indeed, let $x$ be a point in the set on the left hand side of \eqref{A-key-claim}, and consider $B_\rho(x)$. If $\rho\leq \sigma r$, 
then $B_\rho(x)\subset B_{2\sigma r}(y)$ and hence 
\begin{align*}
 \frac{1}{|B_{\rho}(x)|} \int_{B_{\rho}(x)\cap U}{|\nabla  u|^p  dy}
 &\leq \frac{2^{p-1}}{|B_{\rho}(x)|}\Big[ \int_{\Omega_{\rho}(x)}{|\nabla  u -\nabla v|^p dy}
 +\int_{\Omega_{\rho}(x)}{|\nabla  v|^p dy} \Big]\\
 &\leq 2^{p-1} \Big[\M_{\Omega_{2\sigma r}(y)}(|\nabla u - \nabla v|^p)(x)
 +\|\nabla  v\|_{L^\infty(\Omega_{2 \sigma  r}(y))}^p  \Big]\leq 2^p C_*.
\end{align*}
On the other hand, if $\rho> \sigma r$ then $B_\rho(x)\subset B_{3\rho}(x_0)$.
This and the first inequality in \eqref{A-max-one-point} give
\begin{align*}
 \frac{1}{|B_{\rho}(x)|} \int_{B_{\rho}(x)\cap U}{|\nabla  u|^p  dy}
 &\leq  \frac{3^n}{|B_{3\rho}(x_0)|} \int_{B_{3 \rho}(x_0)\cap U}{|\nabla  u|^p  dy}\leq 3^n.
\end{align*}
Therefore, $\M_{U}(|\nabla  u|^p)(x)\leq  N $ and  claim \eqref{A-key-claim} is proved. Notice that \eqref{A-key-claim} is equivalent to 
\begin{equation*}
 \big\{\Omega_{\sigma r}(y):\, \M_{U}(|\nabla  u|^p)>  N \big\} \subset \big\{\Omega_{\sigma r}(y): \, \M_{\Omega_{2\sigma r}(y)}(|\nabla u - \nabla v|^p) > C_*\big\}. 
\end{equation*}
It follows from this, the weak type $1-1$ estimate, and \eqref{A-eq:nabla-u-v} that
\begin{align*}
 \big|\big\{\Omega_{\sigma r}(y):\, \M_{U}(|\nabla  u|^p)>  N \big\}\big|
& \leq \big| \big\{\Omega_{\sigma r}(y): \, \M_{\Omega_{2\sigma r}(y)}(|\nabla u - \nabla v|^p) > C_*\big\} \big|\\
&\leq C \int_{\Omega_{2\sigma r}(y)} |\nabla u - \nabla v|^p dx\leq C_1 \tilde\e^p |B_{\sigma r}(y)|.
\end{align*}
We then infer  from property \eqref{charac-A-infty} that 
\begin{align*}
 w\Big(\big\{\Omega_{\sigma r}(y):\, \M_{U}(|\nabla  u|^p)>  N \big\}\Big)
 &\leq   A \Big(\frac{| \big\{\Omega_{\sigma r}(y):\, \M_{U}(|\nabla  u|^p)>  N \big\}|}{|B_{\sigma r}(y)|}\Big)^{\nu}
w(B_{\sigma r}(y))
\leq A (C_1\tilde\e^{p})^\nu w(B_{\sigma r}(y))
\end{align*}
with   $A$ and $\nu$ being  the  constants given by  characterization \eqref{charac-A-infty} for $w$. We choose $\tilde \e^p  := \min{\{C_1^{-1}(\e A^{-1})^{\frac{1}{\nu}},1\}}$ to complete the proof.
\end{proof}

\begin{lemma}\label{A-boundary-second-density-est}
Assume that   $\A$ satisfies \eqref{structural-reference-1}--\eqref{structural-reference-3}, $\psi\in W^{1,p}(\Omega)$, and $\F\in L^{p'}(\Omega;\R^n)$.
Let  $M>0$ and  $w\in A_s$ for some $1<s<\infty$.  There exists a constant $N=N(p, n,\Lambda)>1$ satisfying  
for  any $\e>0$, we can find   
small positive constants $\delta$  and $\sigma$ depending only on $\e,\, p,\, n,\, \omega,\, \Lambda,\, M, \, s$, and $[w]_{A_s}$
such that:   if $\Omega$ is  $(\delta,R)$-Reifenberg flat,   $\lambda>0$, $\theta>0$, $U\subset \Omega$ is an open set, and
\begin{equation}\label{another-smallness}
\sup_{0<\rho< \frac{R}{2}}\sup_{y\in \overline\Omega} \Theta_{\Omega_{3\sigma\rho}(y)}(\A)\leq  \delta,
\end{equation}
then
  for any weak solution $u$ of \eqref{GE} with
 $\|u\|_{L^\infty(\Omega)}+\|\psi\|_{L^\infty(\Omega)}\leq \frac{M}{\lambda \theta}$,
and for any $y\in \Omega$, $0<r< R/10$ with $\Omega_{21 r}(y)\subset U$ and
\begin{align}
 &B_{\sigma r}(y)\cap  \big\{U:\, \M_{U}(|\nabla  u|^p)\leq 1 \big\}\cap
 \{ U: \M_{U}(|\nabla\psi|^p +|\F|^{p'} )\leq \delta\}\neq \emptyset,
\label{A-boundary-one-point-condition}
 \end{align}
we have
\begin{equation*}
 w\Big( \{ U:\, \M_{U}(|\nabla  u|^p)> N \}\cap B_{\sigma r}(y)\Big)
< \e  \, w(B_{\sigma r}(y)).
\end{equation*}
\end{lemma}
\begin{proof}
We consider the following possibilities:

{\bf Case 1} (away from the boundary): $B_r(y)\subset \Omega$. 
Then  we are in the interior case and $B_r(y)=\Omega_r(y)\subset  U$.
Hence, we obtain the  conclusion by using Lemma~5.1 in
\cite{DiN}. 

{\bf Case 2} (near boundary): $B_{4 \sigma r}(y)\subset \Omega$ but $B_r(y)\cap \partial \Omega\neq \emptyset$.   Let 
$\bar y\in B_{ r}(y)\cap \partial \Omega$. Then 
\begin{equation*}
B_{ r}(y) \subset B_{2  r}(\bar y)\subset B_{3  r}(y).
\end{equation*}
In particular, we have $B_{4 \sigma r}(y)\subset \Omega_{2r}(\bar y)$. 
 Moreover, $\Omega_{4 r}(\bar y)\subset \Omega_{5r}(y) \subset U$ since $B_{4 r}(\bar y)\subset 
B_{5r}(y)$.
Therefore, we can use  Lemma~\ref{A-initial-density-est}
to obtain the desired result.


{\bf Case 3} (boundary): $B_{4\sigma r}(y)\cap \partial \Omega\neq \emptyset$. Let 
$\bar y\in B_{4\sigma r}(y)\cap \partial \Omega$. Then
\begin{equation}\label{A-y-y0}
B_{\sigma r}(y) \subset B_{5 \sigma r}(\bar y)\subset B_{9 \sigma r}(y).
\end{equation}
This together with  assumption \eqref{A-boundary-one-point-condition} yields
\[
B_{5 \sigma r}(\bar y)\cap  \big\{U:\, \M_{U}(|\nabla  u|^p)\leq 1 \big\}\cap
 \{ U: \M_{U}(|\nabla\psi|^p +|\F|^{p'} )\leq \delta\}\neq \emptyset.
 \]
We also have $\Omega_{20 r}(\bar y)\subset \Omega_{21 r}(y) \subset U$ as $B_{20 r}(\bar y)\subset B_{21 r}(y)$.
 Therefore, by applying  Lemma~\ref{A-initial-density-est} for  $y=\bar y$ and
 $\e_1:= 9^{-n s}\e/ [w]_{A_s}$, we obtain 
\begin{equation*}
 w\Big( \{ U:\, \M_{U}(|\nabla  u|^p)> N \}\cap B_{5\sigma r}(\bar y)\Big)
< \e_1   \, w(B_{5 \sigma r}(\bar y)).
\end{equation*}
It follows from this and \eqref{A-y-y0}  that
\begin{equation*}
 w\Big( \{ U:\, \M_{U}(|\nabla  u|^p)> N \}\cap B_{\sigma r}(y)\Big)
< \e_1   \, w(B_{9 \sigma r}(y))
\leq \e_1 9^{ns}[w]_{A_s} w(B_{ \sigma r}(y))
=\e w(B_{ \sigma r}(y)).
\end{equation*}
\end{proof}

Let us fix $\bar R :=R/100$ and a finite collection of points $\{z_i\}_{i=1}^L\subset
\overline \Omega$ such that
$ \displaystyle\overline\Omega\subset \bigcup_{i=1}^L B_{\sigma \bar R}(z_i)$. We also consider an open set $V\subset \Omega$
satisfying: there exists a constant $c_0\in (0,1]$ such that
\begin{equation}\label{V}
|B_{\rho}(z)\cap V |
\geq c_0\, |B_{\rho}(z)| \quad \mbox{for all } z\in  \partial V  \mbox{ and all } 0<\rho<\sigma \bar R.
\end{equation}
In view of Lemma~\ref{A-initial-density-est} and  as in  \cite[Lemma~3.8]{MP2}, we can apply 
Krylov-Safanov lemma, which is  a variation of the Vitali covering lemma, to obtain:
\begin{lemma}\label{A-second-density-est}
Assume that   $\A$ satisfies \eqref{structural-reference-1}--\eqref{structural-reference-3}, $\psi\in W^{1,p}(\Omega)$, and $\F\in L^{p'}(\Omega;\R^n)$. 
Let  $M>0$ and  $w\in A_s$ for some $1<s<\infty$. 
There exists a constant $N=N(p, n,\Lambda)>1$ satisfying  
for  any $\e>0$, we can find   
small positive constants $\delta$  and $\sigma$ depending only on $\e,\, p,\, n,\, \omega,\,\Lambda,\, M,\, s$, and $[w]_{A_s}$
such that:  if $\Omega$ is  $(\delta,R)$-Reifenberg flat, \eqref{another-smallness} holds,  $\lambda>0$, $\theta>0$, and 
$V\subset U\subset\Omega$ are open sets satisfying
\eqref{V} and $\Omega_{\frac{R}{2}}(y)\subset U$ for every $y\in V$, 
then   for any weak solution $u\in W^{1,p}(\Omega)$ of \eqref{GE} with
$\|u\|_{L^\infty(\Omega)}+\|\psi\|_{L^\infty(\Omega)}\leq \frac{M}{\lambda \theta}$  and 
\begin{equation}\label{density-ass}
  w \Big( \{V: 
\M_{U}(|\nabla  u|^p)> N \}\Big) <\e \, w(B_{\sigma \bar R}(z_i)) \quad \forall i=1,2,...,L,
\end{equation}
we have
\begin{align*}
w \Big(\{V: \M_{U}(|\nabla u|^p)> N\}\Big)
\leq \big(\frac{10^n}{c_0}\big)^s [w]_{A_s}^2\, \e \, \Big[
w\Big(\{V: \M_{U}(|\nabla u|^p)> 1\}\Big)
+w\Big(\{ V: \M_{U}(|\nabla\psi|^p +|\F|^{p'})> \delta \}\Big)\Big].\nonumber
\end{align*}
\end{lemma}
\begin{proof}
For $\e>0$, let $N$, $\delta$, and $\sigma$ be the corresponding constants given by Lemma~\ref{A-boundary-second-density-est}. 
 Set
\[
C=\{ V:\, \M_{U}(|\nabla   u|^p)> N \} \quad\mbox{and}\quad  
D=\{ V:\, \M_{U}(|\nabla  u|^p)> 1 \}\cup \{V:\, \M_{U}(|\nabla\psi|^p + |\F|^{p'})> \delta \}.
\]
Let $y$ be any point in $C$, and define 
\[
m(r):= \frac{w(C\cap B_{\sigma r}(y))}{w(B_{\sigma r}(y))} \quad\mbox{for}\quad  r>0.
\]
The lower semicontinuity of $\M_{U}(|\nabla   u|^p)$ implies that $C$ is open, and hence $\lim_{r\to 0^+} m(r)= 1$.
Moreover, 
 as $y\in B_{\sigma \bar R}(z_i)$ for some $i$ we have from  condition  \eqref{density-ass} that
 \[
 m(r) \leq \frac{w(C)}{w(B_{\sigma \bar R}(z_i))} <\e \quad \forall r\geq 2 \bar R.
 \]
 Therefore, there exists $r_y\in (0,2 \bar R)$ such that
$ m(r_y)=\e$ and $m(r)<\e $ for all $r>r_y$. That is,
\begin{equation}\label{A-density-choice}
w(C\cap B_{\sigma r_y}(y))=\e\, w(B_{\sigma r_y}(y))\quad \mbox{and}\quad w(C\cap B_{\sigma r}(y))< \e \, w(B_{\sigma r}(y))\quad \forall r>r_y.
\end{equation}
Thus by Vitali's covering lemma we can select a countable sequence $\{y_i\}_{i=1}^\infty$ such that $\{B_{\sigma r_i}(y_i)\}$ is a sequence of  disjoint balls and
\[
C\subset \bigcup_{i=1}^\infty B_{5\sigma  r_i}(y_i),
\]
where $r_i := r_{y_i}$. Since 
$w\Big( \{ U:\, \M_{U}(|\nabla  u|^p)> N \}\cap B_{\sigma r_i}(y_i)\Big)\geq w(C\cap B_{\sigma r_i}(y_i))=\e\, w(B_{\sigma r_i}(y_i))$ by \eqref{A-density-choice} and $r_i< 2\bar R =R/50$, it follows 
from Lemma~\ref{A-boundary-second-density-est} that
\begin{equation}\label{A-inD}
B_{\sigma r_i}(y_i)\cap V \subset D.
\end{equation}
We have 
\begin{align}\label{A-w-Cballs}
w(C) 
&\leq w\Big(\bigcup_{i=1}^\infty  B_{5 \sigma r_i}(y_i) \cap C\Big)\leq \sum_{i=1}^\infty w(B_{5\sigma  r_i}(y_i) \cap C)\nonumber\\
&\leq \e  \sum_{i=1}^\infty w(B_{5\sigma  r_i}(y_i) )\leq \e [w]_{A_s}  5^{ n s}\sum_{i=1}^\infty w(B_{\sigma  r_i}(y_i) ).
\end{align}
Let $y\in V$ and $0<\rho< 2\sigma \bar R$. If $B_{\frac{\rho}{2}}(y)\subset V$, then
$|B_{\rho}(y)\cap V | \geq |B_{\frac{\rho}{2}}(y)|= 2^{-n}|B_{\rho}(y)|$. Otherwise, there exists $z\in B_{\frac{\rho}{2}}(y)\cap\partial V$.
Then $B_{\frac{\rho}{2}}(z)\subset B_\rho(y)$ and hence it follows from  assumption \eqref{V} that
$|B_{\rho}(y)\cap V |\geq |B_{\frac{\rho}{2}}(z)\cap V |\geq c_0 |B_{\frac{\rho}{2}}(z)| =c_0 2^{-n}|B_{\rho}(y)|$. Combining these,
we conclude that
\[
 \sup_{y\in V, 0<\rho< 2\sigma \bar R} \frac{|B_{\rho}(y)|}{|B_{\rho}(y)\cap V |} \leq \frac{2^n}{c_0}.
\]
This together with property \eqref{strong-doubling}  gives
\[
w(B_{\sigma r_i}(y_i))\leq  [w]_{A_s} \Big( \frac{|B_{\sigma r_i}(y_i)|}{|B_{\sigma r_i}(y_i)\cap V |} \Big)^s 
w(B_{\sigma r_i}(y_i)\cap V) \leq [w]_{A_s} \big(\frac{2^n}{c_0}\big)^s\,
w(B_{\sigma r_i}(y_i)\cap V).
\]
We deduce  from  this and    \eqref{A-inD}--\eqref{A-w-Cballs} that
\[
w(C) \leq  \e [w]_{A_s}^2  \big(\frac{10^n}{c_0}\big)^s\sum_{i=1}^\infty w(B_{ \sigma r_i}(y_i) \cap V)
= 
\e [w]_{A_s}^2 \big(\frac{10^n}{c_0}\big)^s w\Big(\bigcup_{i=1}^\infty B_{\sigma  r_i}(y_i) 
\cap V \Big)\leq \e [w]_{A_s}^2  \big(\frac{10^n}{c_0}\big)^s w(D),
\]
which yields the desired estimate.
\end{proof}
\subsection{Global gradient estimates in weighted $L^q$ spaces}\label{sub:Lebesgue-Spaces}
We are now ready to prove  Theorem~\ref{thm:main} and Theorem~\ref{thm:localized-main}.

\begin{proof}[\textbf{Proof of Theorem~\ref{thm:main}}]
Let  $N=N(p,n,\Lambda)>1$ be as in  Lemma~\ref{A-second-density-est}, and let $l=q/p\geq 1$. We choose 
$\e=\e(p, q, n,\Lambda, s, [w]_{A_s} )>0$ be such that
\[
\e_1 \eqdef 20^{n s } [w]_{A_s}^2\e = \frac{1}{2 N^{l}}, 
\]
and let $\delta$  and $\sigma$ (depending only on $p,\,q,\,n,\,\omega, \,\Lambda,\, M, \,  s$, and $[w]_{A_s}$) be the corresponding 
positive constants given by Lemma~\ref{A-second-density-est}. 
Assume for the moment that 
 $u$ is a weak solution of \eqref{GE} satisfying 
\begin{equation}\label{initial-distribution-condition}
w\big( \{\Omega: 
\M_{\Omega}(|\nabla u|^p)> N \}\big) <  \e \, w(B_{\sigma \bar R}(z_i)) \quad \forall i=1,2,...,L.
\end{equation}
Notice  that $V=\Omega$ satisfies condition \eqref{V}   with $c_0=1/2^n$ since for any 
 $z\in \partial \Omega$ and $0<\rho< \sigma \bar R$ we have from  the Reifenberg flat condition that 
\begin{align*}
|B_{\rho}(z)\cap \Omega |
\geq \big|B_{\rho}(z)\cap \{x: x_n >z_n + \rho \delta\} \big|
\geq \frac{(1-\delta)^n}{n} |B_{\rho}(z)|
\geq \frac{1}{2^n} \, |B_{\rho}(z)|.
\end{align*}
Thus by applying  Lemma~\ref{A-second-density-est} for $V=U=\Omega$ we obtain  
\beq\label{initial-distribution-est}
w\big(\{\Omega: \M_{\Omega}(|\nabla u|^p)> N\}\big)
\leq \e_1  \left[
w\big(\{\Omega: \M_{\Omega}(|\nabla u|^p)> 1\}\big)
+ w\big(\{ \Omega: \M_{\Omega}(|\nabla\psi|^p +|\F |^{p'})> \delta \}\big)\right].
\eeq
 Let us iterate this estimate by considering
\[
u_1(x) = \frac{u(x)}{N^{\frac1p}}, \quad
\psi_1(x) = \frac{\psi(x)}{N^{\frac1p}},\quad  \F_1(x) = \frac{\F(x)}{N^{\frac{p-1}{p}}}\quad \mbox{and}\quad \lambda_1 = N^{\frac1p}\lambda.
\]
It is clear that $\|u_1\|_{L^\infty(\Omega)}+\|\psi_1\|_{L^\infty(\Omega)}\leq \frac{M}{\lambda_1 \theta}$,  and  $u_1\in W^{1,p}(\Omega)$ 
is a weak solution of
$\div \Big[\frac{\A(x,\lambda_1 \theta u_1, \lambda_1 \nabla u_1)}{\lambda_1^{p-1}}  \Big]=\div  \F_1$
in $\Omega$ and $u_1=\psi_1$ on $ \partial\Omega$.
 Moreover, thanks to \eqref{initial-distribution-condition} we have
\begin{align*}
w\big( \{\Omega: 
\M_{\Omega}(|\nabla u_1|^p)> N \}\big) &= w\big( \{\Omega: 
\M_{\Omega}(|\nabla u|^p)> N^2 \}\big) < \e \, w(B_{\sigma \bar R}(z_i)) \quad \forall i=1,2,...,L.
\end{align*}
Therefore, by applying  Lemma~\ref{A-second-density-est} to $u_1$ we get
\begin{align*}
w\big(\{\Omega: \M_{\Omega}(|\nabla u_1|^p)> N\}\big)
&\leq \e_1 \left[
w\big(\{\Omega: \M_{\Omega}(|\nabla u_1|^p)> 1 \}\big)
+ w\big(\{ \Omega: \M_{\Omega}(|\nabla\psi_1|^p +|\F_1 |^{p'} )> \delta \}\big)\right]\\
&= \e_1  \left[ 
w\big(\{\Omega: \M_{\Omega}(|\nabla u|^p)> N \}\big)
+ w\big(\{ \Omega: \M_{\Omega}(|\nabla\psi|^p +|\F|^{p'})> \delta N\}\big) \right].
\end{align*}
We infer from this and  \eqref{initial-distribution-est} that
\begin{align}\label{first-iteration-est}
w\big(\{\Omega: \M_{\Omega}(|\nabla u|^p)> N^2\}\big)
&\leq \e_1^2 
w\big(\{\Omega: \M_{\Omega}(|\nabla u|^p)> 1 \}\big)\\
&\quad + \e_1^2 w\big(\{ \Omega: \M_{\Omega}(G^p)> \delta\} \big)+
\e_1w\big(\{ \Omega: \M_{\Omega}(G^p)> \delta N\}\big), 
\nonumber
\end{align}
where $G^p:=|\nabla\psi|^p +|\F|^{p'}$.
By repeating the iteration, we then conclude that
\begin{align*}\label{decay-distr}
w\big(\{\Omega: \M_{\Omega}(|\nabla u|^p)> N^k\}\big)
&\leq \e_1^k 
w\big(\{\Omega: \M_{\Omega}(|\nabla u|^p)> 1 \}\big)
+ \sum_{i=1}^k\e_1^iw\big(\{ \Omega: \M_{\Omega}(G^p)> \delta N^{k-i}\} \big)\quad \forall k\geq 1.
\end{align*}
This together with  
\begin{align*}
\int_{\Omega}\M_{\Omega}(|\nabla u|^p)^{l} \, dw  
&=l \int_0^\infty t^{l-1} w\big(\{\Omega: \M_{\Omega}(|\nabla u|^p)>t\}\big)\, dt\\
&=l \int_0^{N} t^{l -1} w\big(\{\Omega: \M_{\Omega}(|\nabla u|^p)>t\}\big)\, dt
+l \sum_{k=1}^\infty\int_{N^{k}}^{N^{k+1}} t^{l -1} w\big(\{\Omega: \M_{\Omega}(|\nabla u|^p)>t\}\big)\, dt\\
&\leq N^{l} w(\Omega) + (N^{l} -1) \sum_{k=1}^\infty N^{l k}w\big(\{\Omega: \M_{\Omega}(|\nabla u|^p)>N^k\}\big)
\end{align*}
gives
\begin{align*}
\int_{\Omega}\M_{\Omega}(|\nabla u|^p)^{l} \, dw  
&\leq  N^{l} w(\Omega) + (N^{l} -1)w(\Omega) \sum_{k=1}^\infty (\e_1 N^{l})^k
+\sum_{k=1}^\infty\sum_{i=1}^k (N^{l} -1)N^{l k}\e_1^i w\big(\{ \Omega: 
\M_{\Omega}(G^p)> \delta N^{k-i}\} \big).
\end{align*}
But we have
\begin{align*}
&\sum_{k=1}^\infty\sum_{i=1}^k (N^{l} -1)N^{{l} k}\e_1^iw\big(\{ \Omega: 
\M_{\Omega}(G^p)> \delta N^{k-i}\} \big)\\
&=\big(\frac{N}{\delta}\big)^{l}\sum_{i=1}^\infty (\e_1 N^{l})^i\left[\sum_{k=i}^\infty 
(N^{l} -1)\delta^{{l}}N^{{l} (k-i-1)} w\big(\{\Omega: \M_{\Omega}(G^p)>
\delta N^{k-i}\} \big)\right]\\
&=\big(\frac{N}{\delta}\big)^{l}\sum_{i=1}^\infty (\e_1 N^{l})^i\left[\sum_{j=0}^\infty (N^{l} -1)\delta^{{l}}N^{{l} (j-1)}
w\big(\{ \Omega: \M_{\Omega}(G^p)> \delta N^{j}\} \big)\right]
\leq \big(\frac{N}{\delta}\big)^{l} \Big[\int_{\Omega}\M_{\Omega}(G^p)^{l}  \, dw \Big]
\sum_{i=1}^\infty (\e_1 N^{l})^i.
\end{align*}
Thus we infer that
\begin{align*}
\int_{\Omega}\M_{\Omega}(|\nabla u|^p)^{l} \, dw 
&\leq N^{l} w(\Omega) + \left[ (N^{l} -1) w(\Omega) +\big(\frac{N}{\delta}\big)^{l} 
\int_{\Omega}\M_{\Omega}(G^p)^{l}  \, dw  \right] \sum_{k=1}^\infty (\e_1 N^{l})^k\\
&= N^{l} w(\Omega)+ \left[ (N^{l} -1)w(\Omega) +\big(\frac{N}{\delta}\big)^{l} 
\int_{\Omega}\M_{\Omega}(G^p)^{l}  \, dw \right] \sum_{k=1}^\infty 2^{-k}
\leq C\left( w(\Omega)+ \int_{\Omega}\M_{\Omega}(G^p)^{l} \, dw  \right)
\end{align*}
with the constant  $C$ depending only on $p$, $q$, $n$, $\omega$, $ \Lambda$,     $M$, $s$, and $[w]_{A_s}$. This together with the 
facts $l =q/p$ and  $
|\nabla u(x)|^p  \leq  
\M_{\Omega}(|\nabla u|^p)(x)$
for a.e.  $x\in \Omega$ yields
\begin{equation}\label{initial-desired-L^p-estimate}
\fint_{\Omega}|\nabla u|^q dw \leq C\left( 1+ \fint_{\Omega}\M_{\Omega}(|\nabla\psi|^p +|\F|^{p'})^{\frac{q}{p}} \, dw  \right).
\end{equation}

We next remove the extra assumption  
\eqref{initial-distribution-condition} for $u$. Notice that for any constant $\Upsilon>0$, by using the weak type $1-1$ estimate for the
maximal function  we get
\begin{align}\label{weak1-1}
\big| \{\Omega: 
\M_{\Omega}(|\nabla u|^p)> N  \Upsilon^p\}\big|
\leq \frac{C }{N \Upsilon^p}\int_{\Omega} |\nabla u|^p dx.
\end{align}
We have from \cite[Lemma~8(iii)]{LMS} that $R\leq 4 d$ with $d:=\diam(\Omega)$, which implies that
$\displaystyle\bigcup_{i=1}^L B_{\sigma \bar R}(z_i)\subset \hat B:= B(z_1,  2 d)$. 
Let $\bar{u}(x,t) = u(x,t)/ \Upsilon$, where
\[ \Upsilon^p := \frac{2 C   \|\nabla u\|_{L^p(\Omega)}^p }{N  |\hat B|} \Big[\big(\frac{2 d}{\sigma \bar R}\big)^{n s} \frac{[w]_{A_s} K}{\e} \Big]^{\frac{1}{\beta}}
\]
with
$K$ and $\beta$ being the constants given by Lemma~\ref{weight:basic-pro}. 
Then it follows from \eqref{weak1-1} that
\[\big| \{\Omega: 
\M_{\Omega}(|\nabla \bar{u}|^p)> N\}\big|
\leq  2^{-1} \Big[\big(\frac{\sigma \bar R}{2 d}\big)^{n s} \frac{\e}{[w]_{A_s} K} \Big]^{\frac{1}{\beta}} |\hat B|.
\]
This together with property \eqref{strong-doubling} for $w$ gives
\[w \big( \{\Omega: 
\M_{\Omega}(|\nabla \bar{u}|^p)> N\}\big)
\leq  \big(\frac{\sigma \bar R}{2 d}\big)^{n s} \frac{2^{-\beta} \e}{  [w]_{A_s} } \, w(\hat B)\leq 
2^{-\beta} \e  \, w(B_{\sigma \bar R}(z_i)) \quad \forall i=1,2,...,L.
\]
 Hence we can apply \eqref{initial-desired-L^p-estimate} to $\bar{u}$ with $\F$, $\psi$, and $\lambda$ being replaced by 
$\bar \F= \F/\Upsilon^{p-1}$, $\bar\psi =\psi/\Upsilon$, and $\bar \lambda = \lambda \Upsilon$. By reversing back to the  functions $u$ and  $\F$,  we then  obtain 
\begin{align*}
\fint_{\Omega}|\nabla u|^q dw 
&\leq C\left( \Upsilon^q+ \fint_{\Omega}\M_{\Omega}(|\nabla\psi|^p +|\F|^{p'})^{\frac{q}{p}} \, dw \right)
\leq  C\left( \|\nabla u\|_{L^p(\Omega)}^q+ \fint_{\Omega}\M_{\Omega}(|\nabla\psi|^p +|\F|^{p'})^{\frac{q}{p}} \, dw  \right)
\end{align*}
with  $C$ now also depending on $\bar R=R/100$ and $\diam(\Omega)$.
This yields estimate \eqref{main-estimate} as desired.
\end{proof}

\begin{proof}[{\bf Proof of Theorem~\ref{thm:localized-main}}]
Let  $N=N(p,n,\Lambda)>1$ be as in  Lemma~\ref{A-second-density-est}, and let $l=q/p\geq 1$. We choose 
$\e=\e(p, q, n,\Lambda, s, [w]_{A_s} )>0$ be such that
\[
\e_1 \eqdef 80^{n s } [w]_{A_s}^2\e = \frac{1}{2 N^{l}}, 
\]
and let $\delta$  and $\sigma$ (depending only on $p,\,q,\,n,\,\omega, \,\Lambda,\, M, \,  s$, and $[w]_{A_s}$) be the corresponding 
positive constants given by Lemma~\ref{A-second-density-est}.
Let $y_0\in \overline\Omega$. We first consider the case $r\geq R/2$. Let $U=\Omega_{2r}(y_0)$  and $V=\Omega_{r}(y_0)$. Since $r\geq R/2$, 
we get $\Omega_{\frac{R}{2}}(y)\subset U$ for every $y\in V$. We next verify condition \eqref{V} for $V$ in order to apply  Lemma~\ref{A-second-density-est}. 
For any  $z\in \partial V$ and  $0<\rho< \sigma \bar R$,  we consider the following two possibilities:

{\bf Case 1:} $y_0\in B_\rho(z)$. Then $B_\rho(z)\subset  B_{2\rho}(y_0) \subset B_{r}(y_0)$. In particular, $z\not\in \partial B_{r}(y_0)$ and so we must have $z\in \partial\Omega$. These together with  the Reifenberg flat condition for $\Omega$  give
\begin{align*}
|B_{\rho}(z)\cap V |
=|B_{\rho}(z)\cap \Omega|
\geq \big|B_{\rho}(z)\cap \{x: x_n >z_n + \rho \delta\} \big|
\geq \frac{(1-\delta)^n}{n} |B_{\rho}(z)|
\geq \frac{1}{2^n} \, |B_{\rho}(z)|.
\end{align*}

{\bf Case 2:} $y_0\not\in  B_\rho(z)$.
Then the line  passing through $z$ and $y_0$ intersects  $\partial B_{\rho}(z)$ at two distinct points, say $a_1$ and $a_2$ with $a_1$ 
being the one on the same side as  $y_0$ with respect to the point $z$. As $|a_1 -z| =\rho \leq  |y_0 - z|$, we have in addition that 
$a_1$ belongs to the line segment $[z,y_0]\subset \overline\Omega$ connecting $z$ and $y_0$.
In particular,  $a_1\in B_{r}(y_0)$ since 
$|a_1 - y_0|<|z- y_0|\leq r$.  By letting $w$ be the midpoint of $a_1$ and $z$ we obviously have 
  $w\in \overline\Omega$ and  $B_{\frac{\rho}{2}}(w) \subset B_{\rho}(z)$. Due to  $y_0\not\in  [w,z]$, we have
  $|w- y_0|=|z-y_0|-|z-w|\leq 
  r -\frac{\rho}{2}$. Hence 
       $B_{\frac{\rho}{2}}(w) \subset  B_{r}(y_0)$ as  $x\in B_{\frac{\rho}{2}}(w)$ implies that
   $|x- y_0| \leq |x - w| +|w-y_0|<\frac{\rho}{2} +r -\frac{\rho}{2}=r$. Therefore, we infer that
$B_{\frac{\rho}{2}}(w) \subset B_{\rho}(z)\cap B_{r}(y_0)$ giving 
\begin{equation}\label{w-center}
 |B_{\rho}(z)\cap V |=|B_{\rho}(z)\cap B_r(y_0)\cap \Omega|
\geq |B_{\frac{\rho}{2}}(w)\cap \Omega|.
\end{equation}
If $B_{\frac{\rho}{4}}(w)\subset  \Omega$, then it follows from \eqref{w-center} that $|B_{\rho}(z)\cap V |\geq |B_{\frac{\rho}{4}}(w)|
= 4^{-n} |B_{\rho}(z)|$. Otherwise, there exists $\bar w \in B_{\frac{\rho}{4}}(w)\cap  \partial\Omega$ implying that 
$B_{\frac{\rho}{4}}(\bar w)\subset B_{\frac{\rho}{2}}(w)$. Then by combining with  \eqref{w-center} and  the Reifenberg flat condition we 
obtain
\begin{align*}
|B_{\rho}(z)\cap V |\geq |B_{\frac{\rho}{4}}(\bar w)\cap \Omega|
\geq \big|B_{\frac{\rho}{4}}(\bar w) \cap \{x: x_n >\bar w_n + \frac{\rho}{4} \delta\} \big|
\geq \frac{(1-\delta)^n}{n} |B_{\frac{\rho}{4}}(\bar w)|
\geq \frac{1}{8^n} \, |B_{\rho}(z)|.
\end{align*}
In summary, the above arguments show that $V$ satisfies condition \eqref{V} with $c_0 = 8^{-n}$.
Thus, if 
 $u$ satisfies in addition that
\begin{equation*}\label{A-initial-distribution-condition}
w\big( V: 
\M_{U}(|\nabla u|^p)> N \}\big) <  \e \, w(B_{\sigma \bar R}(z_i)) \quad \forall i=1,2,...,L,
\end{equation*}
then  we can apply   Lemma~\ref{A-second-density-est} to get that
\beq\label{A-initial-distribution-est}
w\big(\{V: \M_{U}(|\nabla u|^p)> N\}\big)
\leq \e_1  \left[
w\big(\{V: \M_{U}(|\nabla u|^p)> 1\}\big)
+ w\big(\{ V: \M_{U}(|\nabla\psi|^p +|\F |^{p'})> \delta \}\big)\right].
\eeq
 Therefore, we can repeat the same arguments as in the proof of Theorem~\ref{thm:main} to   obtain 
\begin{equation*}
\fint_{V}|\nabla u|^q dw 
\leq  C\left( \|\nabla u\|_{L^p(U)}^q+ \fint_{V}\M_{U}(|\nabla\psi|^p +|\F|^{p'})^{\frac{q}{p}} \, dw  \right)
\end{equation*}
with the constant  $C$ depending only on $p$, $q$, $n$, $\omega$, $ \Lambda$,     $M$, $R$, $s$, and $[w]_{A_s}$. 
 Using the definitions of $U$ and $V$,
we infer from this estimate that
\begin{align}\label{eq:universal}
\frac{1}{w(B_r(y_0))}\int_{\Omega_r(y_0)}
&|\nabla u|^q dw 
\leq  C\left(\frac{w(\Omega_r(y_0))}{w(B_r(y_0))} \|\nabla u\|_{L^p(\Omega_{2 r}(y_0))}^q+ \frac{1}{w(B_r(y_0))}\int_{\Omega_r(y_0)}\M_{\Omega_{2 r}(y_0)}(|\nabla\psi|^p +|\F|^{p'})^{\frac{q}{p}} \, dw
\right)\nonumber\\
&\leq  C\left(\|\nabla u\|_{L^p(\Omega_{2 r}(y_0))}^q+ \frac{1}{w(B_r(y_0))}\int_{\Omega_r(y_0)}\M_{\Omega_{2 r}(y_0)}(|\nabla\psi|^p +|\F|^{p'})^{\frac{q}{p}} \, dw  \right)
\quad \,\,\,\forall r\geq \frac{R}{2}. 
\end{align}
We next consider the case  $0<r<R/2$. 
Let  us rescale the problem by setting $\tilde y_0 := r^{-1} y_0$,
$\tilde\Omega := \{r^{-1} x: \, x\in\Omega\}$,  $\tilde w(x) := w(r x )$,  and 
\[
\tilde\A(x, z, \xi) = \A(r x , z,\xi),\quad \tilde\F(x)=\F(r x), \quad \tilde u(x) = r^{-1} u(r x),
\quad \tilde\psi(x) = r^{-1} \psi(r x), \quad\, \,  \tilde\theta = \theta r. 
\]
Then $\tilde u$ is a weak solution of  $\div \Big[\frac{\tilde\A(x,\lambda\tilde\theta \tilde u, \lambda \nabla \tilde u)}{\lambda^{p-1}}\Big] = \div \tilde \F\,$ in 
$\tilde \Omega$ and $\tilde u=\tilde \psi$ on $\partial\tilde \Omega$. We also have
\[
  \Big(\fint_{B_\rho(z)} \tilde w(x)\, dx\Big) \Big(\fint_{B_\rho(z)} \tilde w(x)^{\frac{-1}{s-1}}\, dx\Big)^{s-1}
  =\Big(\fint_{B_{\rho r}(rz )} w(y)\, dy\Big) \Big(\fint_{B_{\rho r}(rz )} 
  w(y)^{\frac{-1}{s-1}}\, dy\Big)^{s-1}
\]
for any ball $B_\rho(z)\subset \R^n$, which implies that
$[\tilde w]_{A_s} = [w]_{A_s}$. Moreover, 
$\tilde\Omega$ is  $(\delta,R /r)$-Reifenberg flat with $R/r>2$, and hence  $\tilde\Omega$ is  
$(\delta,2 )$-Reifenberg flat.
Therefore, we can apply  estimate \eqref{eq:universal} for $R=2$ and for 
$\tilde u$ and weight $\tilde w(x)$ to obtain
\begin{equation}\label{radius=1-2}
\frac{1}{\tilde w (B_1(\tilde y_0))} \int_{\tilde\Omega_1(\tilde y_0)}|\nabla \tilde u|^q d\tilde w 
\leq  C\left[ \|\nabla \tilde u\|_{L^p(\tilde\Omega_2(\tilde y_0))}^q+ \frac{1}{\tilde w (B_1(\tilde y_0))} 
\int_{\tilde\Omega_1(\tilde y_0)}\M_{ \tilde\Omega_2(\tilde y_0)}(
|\nabla\tilde\psi|^p +|\tilde \F|^{p'})^{\frac{q}{p}} \, d\tilde w  \right].
\end{equation}
Since
\begin{align*}
 &\tilde w (B_1(\tilde y_0))=\int_{B_1(\tilde y_0)} w(r x)\, dx = r^{-n} \,  w (B_{r}(y_0)),\quad
 \int_{\tilde\Omega_2(\tilde y_0)} |\nabla \tilde u|^p\, dx = r^{-n}\int_{\Omega_{2 r}(y_0)}|\nabla u|^p\, dy,\\
 &\M_{\tilde\Omega_2(\tilde y_0)}(
|\nabla\tilde\psi|^p +|\tilde \F|^{p'})(x) = \M_{\Omega_{2 r}( y_0)}(
|\nabla\psi|^p +| \F|^{p'})(r x),
\end{align*}
by changing variables we see that \eqref{radius=1-2} is  equivalent to
\begin{equation*}
\frac{1}{w (B_{r}(y_0))} \int_{\Omega_{r}(y_0)}|\nabla u|^q d w 
\leq  C\left[ \Big(r^{-n}\int_{\Omega_{2 r}(y_0)}|\nabla u|^p\, dx\Big)^{\frac{q}{p}}+ \frac{1}{ w (B_{r}(y_0))} 
\int_{\Omega_{r}(y_0)}\M_{\Omega_{2 r}(y_0)}(
|\nabla\psi|^p +| \F|^{p'})^{\frac{q}{p}} \, d w  \right]
\end{equation*}
for $0<r<R/2$. This estimate together with \eqref{eq:universal} gives the conclusion of the theorem. 
We note that unlike the situation in \eqref{eq:universal}, the constant $C$ is  independent of 
$R$ when   $r< R/2$.  
\end{proof}
\begin{remark}
 It is important to stress that in  order to derive  the above estimate for one particular  region $\Omega_r(y)$, we only need to assume 
 $u=\psi$ on the portion $\partial\Omega\cap B_{2r}(y)$.
\end{remark}

\subsection{Global gradient estimates in weighted Morrey spaces}\label{sec:Morrey-Spaces}
In this subsection we present the proofs of Theorem~\ref{thm:weighted-morrey} and Theorem~\ref{thm:global-weighted-morrey}. 
 \begin{proof}[\textbf{Proof of Theorem~\ref{thm:weighted-morrey}}]
  Let $G^p:= |\psi|^p + |\F|^{p'}$.  Let $\bar x\in \Omega$ and $0<r\leq d:=\diam(\Omega)$.
Then by   applying  the localized estimate in 
 Theorem~\ref{thm:localized-main}  we obtain
\begin{align}\label{eq:localized-est}
\frac{1}{|w(B_r(\bar x))|}\int_{B_r(\bar x)\cap\Omega}|\nabla u|^q  dw
&\leq  C\Bigg[  \Big(\frac{1}{|B_{2r}(\bar x)|}\int_{B_{2 r}(\bar x)\cap\Omega} |\nabla u|^p dx\Big)^{\frac{q}{p}} + 
\frac{1}{|w(B_{r}(\bar x))|} \int_{B_r(\bar x)\cap\Omega}\M_{\Omega}(G^p)^{\frac{q}{p}}  dw \Bigg].
\end{align}
We next estimate the first term in the above right hand side. For this, let  $\e\in (0,n)$ to be determined later and  use the trick in  \cite[Page~2506]{MP2} to    write 
\begin{align*}
\frac{1}{|B_{2r}(\bar x)|}\int_{B_{2 r}(\bar x)\cap\Omega} |\nabla u|^p dx
=\omega_n^{-1}
(2 r)^{-\e}\int_{B_{2 r}(\bar x)\cap\Omega}|\nabla u|^p \bar w \, dx
\leq \omega_n^{-1} 
(2 r)^{-\e}\int_{ \Omega}|\nabla u|^p \bar w \, dx
\end{align*}
with $\omega_n :=|B_1|$ and  $\bar w$  being  the weight defined by 
\[
\bar w(x) := \min{\{|x - \bar x|^{-n+ \e  }, (2 r)^{-n +\e }\} }.
\]
As $\bar w\in A_t$  with 
$[\bar w]_{A_t} \leq C(t,\e,n)$ for any $1<t<\infty$ (see \cite[Lemma~3.2]{MP2}), we can apply Theorem~\ref{thm:main} with $q=p$  to estimate the above last integral.
As a consequence, we obtain 
\begin{align}\label{Mo2}
 \frac{1}{|B_{2r}(\bar x)|}\int_{B_{2 r}(\bar x)\cap\Omega} |\nabla u|^p dx
&\leq  C(2 r)^{-\e}\Bigg( \bar w(\Omega)\, \|\nabla u\|_{L^p(\Omega)}^p  +  \int_{\Omega} \M_{\Omega}(G^p) \, 
\bar w\, dx \Bigg)\nonumber\\
&\leq  C (2 r)^{-\e}\Bigg( \|\nabla u\|_{L^p(\Omega)}^p  +  \int_{\Omega} \M_{\Omega}(G^p) \, 
\bar w\, dx \Bigg)
\end{align}
with $C>0$  
 depending only on  $p$, $n$, $\omega$, $\Lambda$, $M$, $R$, $\diam(\Omega)$, and $\e$.
Notice that to obtain the last inequality we have used the fact
\[
\bar w(\Omega) \leq \int_{B_{d}(\bar x)} |x - \bar x|^{-n+\e } \, dx
=\omega_n \int_0^{d} t^{\e  -1 } dt
=\frac{ \omega_n }{\e} \, d^{\e}. 
\]
To bound  the last integral in \eqref{Mo2}, we employ    Fubini's theorem to get
 \begin{align*}
 \int_{\Omega} \M_{\Omega}(G^p)  \,  \bar w\, dx 
 &=\int_0^\infty \int_{\{\Omega: \bar w(x)>t\}} \M_{\Omega}(G^p)  \, dx dt
 \leq \int_0^{ (2 r)^{-n +\e }}   \int_{B_{t^{(-n +\e)^{-1}}}(\bar x)\cap \Omega}\M_{\Omega}(G^p)  \, dx dt\\
 &\leq \int_0^{ d^{-n +\e }}   \int_{\Omega}\M_{\Omega}(G^p)  \, dx dt
 +  \int_{d^{-n +\e }}^{ (2 r)^{-n +\e }}   \int_{B_{t^{(-n +\e)^{-1}}}(\bar x)\cap \Omega}\M_{\Omega}(G^p)  \, dx dt.
  \end{align*}
Since  $\Omega= B_{d}(\bar x)\cap \Omega$, we  then deduce that
    \begin{align*}
 \int_{\Omega} \M_{\Omega}(G^p)   \bar w\, dx 
 &\leq  C \| \M_{\Omega}(G^p)\|_{\calM^{1,\varphi^{\frac{p}{q}}}(\Omega)} 
 \left[ d^\e \,\, \varphi(B_{d}(\bar x))^{\frac{-p}{q}}  +  \int_{d^{-n +\e }}^{ (2 r)^{-n +\e }} 
t^{\frac{n}{-n +\e}} \varphi(B_{t^{(-n +\e)^{-1}}}(\bar x))^{\frac{-p}{q}}  dt\right],
  \end{align*}
  where we recall that $\calM^{1,\varphi}(U)$ denotes  the  Morrey space $\calM^{1,\varphi}_w(U)$ with $w=1$.  As $\varphi
  \in \mathcal B_+$
  by the assumption and $\{\mathcal B_\alpha\}$ is decreasing in $\alpha$, there exists $\alpha\in (0,n)$ such that  $\varphi\in
  \mathcal B_\alpha$.   It then follows if $\e< \alpha p/q$ that 
  \begin{align*}
 \int_{\Omega} \M_{\Omega}(G^p)  \bar w\, dx 
 &\leq   C (2 r)^{\alpha \frac{p}{q}}  \|\M_{\Omega}(G^p)\|_{\calM^{1,\varphi^{\frac{p}{q}}}(B_{\frac{15}{2}})} 
 \Bigg[d^{\e-\alpha \frac{p}{q}}  \varphi(B_{ r}(\bar x))^{\frac{-p}{q}}  + 
 \varphi(B_{2 r}(\bar x))^{\frac{-p}{q}}  \int_0^{ (2 r)^{-n +\e }} 
t^{\frac{n-\alpha \frac{p}{q}}{-n +\e}}   dt\Bigg]\\
 &\leq   C   r^{\e} \varphi(B_r(\bar x))^{\frac{-p}{q}}
 \|\M_{\Omega}(G^p)\|_{\calM^{1,\varphi^{\frac{p}{q}}}(\Omega)}.
   \end{align*}
 Combining this with \eqref{Mo2}, we arrive at:
\begin{align*}
  \frac{1}{|B_{2r}(\bar x)|}\int_{B_{2 r}(\bar x)\cap\Omega} |\nabla u|^p dx
\leq   C\Bigg(r^{-\e} \|\nabla u\|_{L^p(\Omega)}^p+ \varphi(B_r(\bar x))^{\frac{-p}{q}}
\| \M_{\Omega}(G^p)\|_{\calM^{1,\varphi^{\frac{p}{q}}}(\Omega)} \Bigg).
\end{align*}
Therefore, we infer from \eqref{eq:localized-est} and the fact $\varphi\in \mathcal B_\alpha$  that 
\begin{align*}
\frac{\varphi(B_r(\bar x)) }{|w(B_r(\bar x))|}\int_{B_r(\bar x)\cap\Omega}|\nabla u|^q  \, dw
&\leq  C_\e\Bigg[\varphi(B_r(\bar x)) r^{\frac{-\e q}{p}}\|\nabla u\|_{L^p(\Omega)}^q
+\| \M_{\Omega}(G^p)\|_{\calM^{1,\varphi^{\frac{p}{q}}}(\Omega)}^{\frac{q}{p}}
+
\|\M_{\Omega}(G^p)\|_{\calM^{\frac{q}{p},\varphi}_w(\Omega)}^{\frac{q}{p}}\Bigg]\\
&\leq  C_\e\Bigg[\varphi(B_{d}(\bar x)) r^{\alpha-\frac{\e q}{p}}\|\nabla u\|_{L^p(\Omega)}^q
+\| \M_{\Omega}(G^p)\|_{\calM^{1,\varphi^{\frac{p}{q}}}(\Omega)}^{\frac{q}{p}}
+
\|\M_{\Omega}(G^p)\|_{\calM^{\frac{q}{p},\varphi}_w(\Omega)}^{\frac{q}{p}}\Bigg]
\end{align*}
for all $\bar x\in \Omega$ and $0<r\leq d$.  By taking $\e = \frac{\alpha}{2}\frac{p}{q}$
and as $\sup_{\bar x\in \Omega} \varphi( B_{d}(\bar x))<\infty$, this gives estimate \eqref{weighted-morrey-est}.
   \end{proof}
  
  \begin{proof}[\textbf{Proof of Theorem~\ref{thm:global-weighted-morrey}}]
  Let $G^p:= |\psi|^p + |\F|^{p'}$.  Since  $v\in A_{\frac{q}{p}}$, 
  Lemma~2.1 in \cite{DiN} gives  
    $ v^{1-(\frac{q}{p})'} \in A_{(\frac{q}{p})'}\subset A_\infty$. Thus our assumptions imply that condition  ({\bf B}) 
  in  Lemma~\ref{lm:two-weight-case} is satisfied.
Also as $\varphi \in \mathcal B_0$, it is clear that \eqref{cond:vwvarphi}  yields \eqref{wmu1-vmu2-sufficient}. Indeed, for any 
$y\in \R^n$ and any 
$s\geq 2r>0$ we have from \eqref{cond:vwvarphi} and $\varphi \in \mathcal B_0$ that
\[
  \frac{v(B_s(y))}{w(B_s(y))} \frac{1}{\phi(B_s(y))} 
  \leq C_* \, \frac{1}{\varphi(B_{\frac{s}{2}}(y))} \leq C_* C \, \frac{1}{\varphi(B_r(y))}
\]
yielding \eqref{wmu1-vmu2-sufficient}.
Moreover, by  \cite[Theorem~9.3.3]{Gra} there exist $s\in (1,\infty)$ and $C>0$ depending only on $n$ and $[w]_{A_\infty}$ such that $[w]_{A_s}\leq C$.
 Therefore, it follows  from Theorem~\ref{thm:weighted-morrey} and  Lemma~\ref{lm:two-weight-case} that
\begin{equation}\label{eq:1.3+3.4}
\| \nabla u \|_{\calM^{q,\varphi}_w(\Omega)}  \leq    C\Bigg(
\|\nabla u\|_{L^p(\Omega)}
+ \| \M_{\Omega}(G^p)\|_{\calM^{1, \varphi^{\frac{ p}{q}}}(\Omega)}^{\frac{1}{p}}
+  \| G\|_{\calM^{q,\phi}_v(\Omega)}\Bigg)
\end{equation}
with $C>0$ depending  only on  $q$, $p$, $n$, $\omega$, $\Lambda$,   $M$, $R$, $\diam(\Omega)$,  $\varphi$, $C_*$, $[w]_{A_\infty}$, $[v]_{A_{\frac{q}{p}}}$,  and $[w, v^{1-(\frac{q}{p})'}]_{A_{\frac{q}{p}}}$.
Thus it remains to estimate the middle term on the right hand side of \eqref{eq:1.3+3.4}. Let $l:=q/p>1$. Then 
for  any nonnegative function $g\in L^1(\Omega)$,  we obtain  from   H\"older inequality and assumption  \eqref{cond:wv}
   that
  \begin{align*}
 \frac{\varphi(B_R(\bar x))}{|B_R(\bar x)|^l} \Big(\int_{B_R(\bar x)\cap \Omega} g \, dx\Big)^l
&\leq \frac{\varphi(B_R(\bar x))}{|B_R(\bar x)|^l} \Big( \int_{B_R(\bar x)\cap \Omega} g^l  v\, dx\Big) 
 \Big(\int_{B_R(\bar x)} v^{1 - l'} \Big)^{l-1}\\
 &\leq  [w, v^{1-l'}]_{A_l}  \frac{\varphi(B_R(\bar x))}{w\big( B_R(\bar x)\big)}  \int_{B_R(\bar x)\cap\Omega} g^l  v\, dx
 \leq   [w, v^{1-l'}]_{A_l} \|g\|_{\calM^{l,\hat\varphi}_v(\Omega)}^l\nonumber
  \end{align*}
  for all $\bar x\in \Omega$ and all $0<R\leq \diam( \Omega)$, where
  \[
  \hat\varphi(B) :=  \frac{v(B)}{w(B)} \varphi(B).
  \]
  Hence we infer  that
 \begin{equation}\label{eq:est-in-terms-v}
   \| \M_{\Omega}(G^p)\|_{\calM^{1, \varphi^{\frac{1}{l}}}(\Omega)}\leq 
    [w, v^{1-l'}]_{A_l}^{\frac{1}{l}} \,\, \| \M_{\Omega}(G^p)\|_{\calM^{l,\hat\varphi}_v(\Omega)}.
  \end{equation}
Using  $\phi\in \mathcal B_0$, condition \eqref{cond:vwvarphi}, and the doubling property of $w$ due to Lemma~\ref{weight:basic-pro}, we have
\begin{align*}
\sup_{s\geq 2r}  \frac{1}{\phi(B_s(y))} \leq
C \frac{1}{\phi(B_{2r}(y))}\leq C C_* \frac{w(B_{2r}(y))}{v(B_{2r}(y))}
\frac{1}{\varphi(B_r(y))}\leq C' \frac{w(B_r(y))}{v(B_r(y))}
\frac{1}{\varphi(B_r(y))} =C'\frac{1}{\hat\varphi(B_r(y))}
\end{align*} 
for all $y\in \R^n$ and $r>0$.    
 Thus as $v\in A_l$ we can use the strong type estimate for   the Hardy--Littlewood   maximal function 
 given by Lemma~\ref{lm:one-weight-case} to estimate the right hand side of \eqref{eq:est-in-terms-v}. As a result, we get
 \begin{equation*}\label{eq:option-1}
   \| \M_{\Omega}(G^p)\|_{\calM^{1, \varphi^{\frac{1}{l}}}(\Omega)}\leq C \|G^p\|_{\calM^{l,\phi}_v(\Omega)}
    = C \| G\|_{\calM^{q,\phi}_v(\Omega)}^p.
  \end{equation*}
 This and \eqref{eq:1.3+3.4} yield  desired estimate
\eqref{eq:neat-global-est}.
\end{proof}

\begin{appendices}
\section{A compactness argument}
\begin{lemma}\label{lm:compactness}
Let $\ba:\R^n\to\R^n$  be a continuous vector field such that 
\eqref{structural-consequence} holds and $|\ba(\xi)| \leq \Lambda |\xi|^{p-1}$
for some  constants $p>1$ and $\Lambda>0$.
Then for any  $\e>0$,  there exists a constant $\delta>0$ depending only on $\e$, $p$, $\Lambda$,  and  $n$ satisfying: 
if $\Omega\subset\R^n$ is  a bounded domain with
\begin{align*}
 B_{3} \cap \{x_n> 3\delta\}  \subset \Omega_{3} \subset B_{3}\cap \{x_n> -3 \delta\},
\end{align*}
$w\in W^{1,p}(\Omega_3)$ is a weak solution of 
\begin{equation*} 
 \left\{
 \begin{alignedat}{2}
  &\div \, \ba(\nabla w) =0\quad \mbox{in}\quad \Omega_{3},\\\
&w=0\qquad\qquad\mbox{on}\quad B_3 \cap  \partial \Omega
 \end{alignedat} 
  \right.
\end{equation*}
with  $ \frac{1}{|B_3|}\int_{\Omega_3} |\nabla w|^p\, dx\leq 1$,
then there exists  a weak solution $v$ of
\begin{equation*} 
 \left\{
 \begin{alignedat}{2}
  &\div \, \ba(\nabla  v)=0\quad \mbox{in}\quad B_{3}^+,\\\
& v=0 \qquad\qquad \mbox{ on}\quad B_{3} \cap \{x_n=0\}
 \end{alignedat} 
  \right.
\end{equation*}
satisfying  $ \frac{1}{|B_3|}\int_{B_3^+} |\nabla v|^p\, dx \leq (4^p\Lambda^2 )^p$ such that 
\begin{equation*}
\Big(\int_{B_{3} \cap \{x_n> 3\delta\} } |w- v|^p\, dx \Big)^{\frac1p}\leq \e^p.
\end{equation*} 
\end{lemma}
\begin{proof} 
Assume to the contrary that the statement is false. Then there exist  $\e_0>0$, $n\in\N$, $p>1$,  $\Lambda>0$, and sequences $\{\ba^k\}$, 
$\{\Omega^k\}$, $\{w^k\}$ such that for each $k\in\N $ we have:  $\ba^k$ satisfies \eqref{structural-consequence} and $|\ba^k(\xi)| \leq \Lambda |\xi|^{p-1}$, 
\begin{align*}
 B_{3} \cap \{x_n> \frac3k\}  \subset \Omega^k_{3} \subset B_{3}\cap \{x_n> -\frac3k\},
\end{align*}
$w^k\in W^{1,p}(\Omega^k_3)$ is a weak solution of 
  \begin{equation*} 
 \left\{
 \begin{alignedat}{2}
  &\div \, \ba^k(\nabla w^k) =0\quad \mbox{in}\quad \Omega^k_{3},\\\
&w^k=0\qquad\qquad\mbox{on}\quad B_3 \cap  \partial \Omega^k
 \end{alignedat} 
  \right.
\end{equation*}
with  $ \int_{\Omega^k_3} |\nabla w^k|^p\, dx\leq |B_3|$, and 
  \begin{equation}\label{contradic-ass}
\Big(\int_{B_{3} \cap \{x_n> \frac3k\} } |w^k- v|^p\, dx \Big)^{\frac1p}> \e_0^p
\end{equation}
for every weak solution $v$ of
\begin{equation*} 
 \left\{
 \begin{alignedat}{2}
  &\div \,\ba^k(\nabla  v)=0\quad \,\,\mbox{in}\quad B_{3}^+,\\\
& v=0 \qquad\qquad\quad \mbox{on}\quad B_{3} \cap \{x_n=0\}
 \end{alignedat} 
  \right.
\end{equation*}
satisfying $ \frac{1}{|B_3|}\int_{B_3^+} |\nabla v|^p\, dx \leq (4^p\Lambda^2 )^p$.
Notice that $\|w^k\|_{W^{1,p}(\Omega^k_3)}\leq C$ by using Pointcar\'e inequality. 
Then  as in \cite{NP,BR} we can show that there exist a continuous vector field
$ \ba:\R^n\to \R^n$   and a function $w\in W^{1,p}(B_3^+)$ with $ \int_{B_3^+} |\nabla w|^p\, dx\leq |B_3|$ such that up to a subsequence we have
$\ba^k(\xi)\to \ba(\xi)$ for all $\xi\in \R^n$, $w^k\to w$ strongly in $L^p_{loc}(B_3^+)$, and $\nabla w^k\to \nabla w$ weakly in $L^p_{loc}(B_3^+)$. 
Consequently,  $\ba$ satisfies  \eqref{structural-consequence} and  $|\ba(\xi)| \leq \Lambda |\xi|^{p-1}$.
We then infer by passing to the limit that 
 $w\in W^{1,p}(B_3^+)$ is a weak solution of  the equation 
\begin{equation}\label{limit-eq}
 \left\{
 \begin{alignedat}{2}
  &\div \, \ba(\nabla w)=0\quad \mbox{in}\quad B_{3}^+,\\\
& w=0 \qquad\qquad \mbox{ on}\quad B_{3} \cap \{x_n=0\}.
 \end{alignedat} 
  \right.
\end{equation}
Now for each $k\in \N$, let $v^k$ be a weak solution of 
\begin{equation*} 
 \left\{
 \begin{alignedat}{2}
  &\div \, \ba^k(\nabla v^k)=0\quad \mbox{in}\quad B_{3}^+,\\\
& v^k=w \qquad\qquad \mbox{ on}\quad \partial B_{3}^+.
\end{alignedat}
  \right.
\end{equation*}
In particular, we have $v^k=0$ on $  B_{3} \cap \{x_n=0\}$. Moreover, by using $v^k-w$ as a test function and due to 
the structural conditions for  $\ba^k$ we get
\begin{align*}
 \int_{B_3^+} |\nabla v^k|^p dx\leq 4^p\Lambda^2 \int_{B_3^+} |\nabla v^k|^{p-1} |\nabla w| dx 
 \leq   \frac{1}{p'}\int_{B_3^+} |\nabla v^k|^p  dx + \frac1p (4^p\Lambda^2 )^p \int_{B_3^+} |\nabla w|^p  dx
\end{align*}
yielding $\int_{B_3^+} |\nabla v^k|^p dx\leq (4^p\Lambda^2 )^p |B_3|$.
Hence it follows from \eqref{contradic-ass} that 
 \begin{equation}\label{cons-contradict}
\Big(\int_{B_{3} \cap \{x_n> \frac3k\} } |w^k- v^k|^p\, dx \Big)^{\frac1p}> \e_0^p.
\end{equation}
Notice that Pointcar\'e  inequality implies that the sequence $\{v^k\}$ is bounded in $W^{1,p}(B_3^+)$.
Thus there exists a function $v\in W^{1,p}(B_3^+)$ such that up to a subsequence we have 
$v^k\to v$ strongly in $L^p(B_3^+)$ and $\nabla v^k\to \nabla v$ weakly in $L^p(B_3^+)$. By passing to the limit we see that 
 $v\in W^{1,p}(B_3^+)$ is a weak solution of equation \eqref{limit-eq}.
 But as \eqref{limit-eq} has a unique weak solution since $\ba$ satisfies  \eqref{structural-consequence}, we infer that $ v\equiv w$. Therefore, 
 $w^k-v^k\to w-w=0$ strongly in $L^p_{loc}(B_3^+)$. This contradicts 
\eqref{cons-contradict} and the proof is complete.
\end{proof}

\end{appendices}

  \end{document}